\documentclass[11pt]{article}
\usepackage{etex}
\reserveinserts{28}
\usepackage[all]{xy}

\usepackage{amsfonts}
\usepackage{amsthm}
\usepackage{enumerate}
\usepackage{graphicx}
\usepackage{mathrsfs}
\usepackage{bm}
\usepackage{cite}
\usepackage[colorlinks,linkcolor=blue,anchorcolor=blue,citecolor=blue]{hyperref}
\usepackage{amssymb,amsmath} 
\usepackage{indentfirst}
\usepackage{makecell}
\usepackage{tikz}
\usepackage{pgf}
\usetikzlibrary{patterns}
\usepackage{pgffor}
\usepackage{pgfcalendar}
\usepackage{pgfpages}
\usepackage{shuffle,yfonts}
\usepackage{mathtools}
\DeclareFontFamily{U}{shuffle}{}
\DeclareFontShape{U}{shuffle}{m}{n}{ <-8>shuffle7 <8->shuffle10}{}

\newcommand{\A}{{\rm A}}

\newcommand{\ola}{\overleftarrow}
\newcommand{\ora}{\overrightarrow}

\newcommand{\bfj}{{\boldsymbol{\sl{j}}}}
\newcommand{\bfk}{{\boldsymbol{\sl{k}}}}
\newcommand{\bfa}{{\boldsymbol{\sl{a}}}}
\newcommand{\bfl}{{\boldsymbol{\sl{l}}}}
\newcommand{\bfm}{{\boldsymbol{\sl{m}}}}

\newcommand{\bfx}{{\boldsymbol{\sl{x}}}}
\newcommand{\bfz}{{\boldsymbol{\sl{z}}}}

 \allowdisplaybreaks
\usetikzlibrary{arrows,shapes,chains}

\catcode`!=11
\let\!int\int \def\int{\displaystyle\!int}
\let\!lim\lim \def\lim{\displaystyle\!lim}
\let\!sum\sum \def\sum{\displaystyle\!sum}
\let\!sup\sup \def\sup{\displaystyle\!sup}
\let\!inf\inf \def\inf{\displaystyle\!inf}
\let\!cap\cap \def\cap{\displaystyle\!cap}
\let\!max\max \def\max{\displaystyle\!max}
\let\!min\min \def\min{\displaystyle\!min}
\catcode`!=12

\let\oldsection\section
\renewcommand\section{\setcounter{equation}{0}\oldsection}

\allowdisplaybreaks

\DeclareMathOperator*{\dep}{dep}
\DeclareMathOperator{\Li}{Li}

\def\R{\mathbb{R}}

\def\N{\mathbb{N}}
\def\Z{\mathbb{Z}}
\def\Q{\mathbb{Q}}

\def\ze{\zeta}

\theoremstyle{plain}
\newtheorem{thm}{Theorem}[section]
\newtheorem{lem}[thm]{Lemma}
\newtheorem{cor}[thm]{Corollary}

\theoremstyle{definition}

\newtheorem{re}[thm]{Remark}

\begin{document}
\title{\bf Parametric Ap\'{e}ry-type Series and Hurwitz-type Multiple Zeta Values}
\author{
{Masanobu Kaneko${}^{a,}$\thanks{Email: mkaneko@math.kyushu-u.ac.jp (M. Kaneko), ORCID 0000-0002-6658-7313},\ Weiping Wang${}^{b,}$\thanks{Email: wpingwang@zstu.edu.cn (W. Wang), ORCID 0009-0001-5162-0598}, \ Ce Xu${}^{c,}$\thanks{Email: cexu2020@ahnu.edu.cn (C. Xu),\ {\bf corresponding author}, ORCID 0000-0002-0059-7420}\ and\ Jianqiang Zhao${}^{d,}$\thanks{Email: zhaoj@ihes.fr (J. Zhao), ORCID 0000-0003-1407-4230}}\\[1mm]
\small a. Faculty of Mathematics, Kyushu University
744 Motooka, Nishi-ku, Fukuoka, 819-0395, Japan\\
\small b. School of Science, Zhejiang Sci-Tech University, Hangzhou 310018, P.R. China\\
\small c. School of Mathematics and Statistics, Anhui Normal University, Wuhu 241002, P.R. China\\
\small d. Department of Mathematics, The Bishop's School, La Jolla, CA 92037, United States of America
}

\date{}
\maketitle

\noindent{\bf Abstract.} In this paper, we extend the main results of a 2024 \emph{Advances in Applied Mathematics} paper \cite{XuZhao2021c} about Ap\'{e}ry-type series involving central binomial coefficients and the multiple ($t-$)harmonic sums to parametric Ap\'{e}ry-type series involving parametric binomial coefficients and Hurwitz-type multiple harmonic (star) sums. In particular, we will establish many explicit relations between parametric Ap\'{e}ry-type series involving one or two parametric binomial coefficients and Hurwitz-type multiple zeta values (with $r$-variables) by using the method of iterated integrals.



\section{Introduction}
We begin with some basic notations. Let $\N$, $\N^-$, $\Q$, $\R$ and $\mathbb{C}$ be the set of positive integers, negative integers, rational numbers, real numbers and complex numbers, respectively, and $\N_0:=\N\cup \{0\}$ and $\N_0^-:=\N^-\cup \{0\}$.
A finite sequence $\bfk:=(k_1,\ldots, k_r)\in\N^r$ is called a \emph{composition}. We put
\begin{equation*}
 |\bfk|:=k_1+\cdots+k_r,\quad \dep(\bfk):=r,
\end{equation*}
and call them the weight and the depth of $\bfk$, respectively. If $k_1>1$, $\bfk$ is called \emph{admissible}.
As a convention, we denote by $\{m\}_r$ the sequence of $m$'s with $r$ repetitions.

For a composition $\bfk=(k_1,\ldots,k_r)$ and positive integer $n$, the \emph{multiple harmonic sums} and \emph{multiple harmonic star sums} are defined by
\begin{align}
\zeta_n(\bfk):=\sum\limits_{n\geq n_1>\cdots>n_r>0 } \frac{1}{n_1^{k_1}\cdots n_r^{k_r}}\in \Q\quad
\text{and}\quad
\zeta^\star_n(\bfk):=\sum\limits_{n\geq n_1\geq\cdots\geq n_r>0} \frac{1}{n_1^{k_1}\cdots n_r^{k_r}}\in \Q\label{MHSs+MHSSs},
\end{align}
respectively. If $n<r$ then ${\zeta_n}(\bfk):=0$ and ${\zeta _n}(\emptyset )={\zeta^\star _n}(\emptyset ):=1$. When $\bfk$ is admissible, by taking the limit $n\rightarrow \infty$ in \eqref{MHSs+MHSSs} we get the \emph{multiple zeta values} (MZVs) and the \emph{multiple zeta star values}, respectively
\begin{align*}
{\zeta}( \bfk):=\lim_{n\rightarrow \infty}{\zeta _n}(\bfk)\in \R \quad
\text{and}\quad
{\zeta^\star}( \bfk):=\lim_{n\rightarrow \infty}{\zeta^\star_n}( \bfk)\in \R.
\end{align*}
When the depth is 1, we recover the classical Riemann zeta values $\zeta(n)$.

As is well known, the transcendence of $\zeta(2n+1)$ is still an unsolved problem. In 1979, Ap\'ery \cite{Apery1978} shocked the math world by showing that $\zeta(3)$ is irrational. His key idea is to use the series expansion of the form
\begin{equation}\label{equ:Apery3}
\zeta(3)=\frac52 \sum_{k=1}^\infty \frac{(-1)^{k+1}}{k^3\binom{2k}{k}}.
\end{equation}
We say a series is of \emph{Ap\'ery-type} if it is obtained from a (variant of the) MZV series by multiplying or dividing a factor of some binomial coefficient. In this paper, we will mainly study these series, but often with the binomial coefficient appearing in the numerator, contrary to what Ap\'ery did in \eqref{equ:Apery3} or the series
considered in \cite{BBB2006,BorweinBr1997}.

The systematic study of MZVs began in the early 1990s with the works of Hoffman \cite{H1992} and Zagier \cite{DZ1994}. Due to their surprising and sometimes mysterious appearance in the study of many branches of mathematics and theoretical physics, these special values have attracted a lot of attention and interest in the past three decades (for example, see the book by the last author \cite{Zhao2016}).

In general, let $\bfk=(k_1,\ldots,k_r)\in\N^r$ and $\bfz=(z_1,\dotsc,z_r)$ where $z_1,\dotsc,z_r$ are the $N$th roots of unity. We can define the \emph{colored MZVs} of level $N$ by
\begin{equation}\label{equ:defnMPL}
\Li_{\bfk}(\bfz):=\sum_{n_1>\cdots>n_r>0}
\frac{z_1^{n_1}\cdots z_r^{n_r}}{n_1^{k_1} \cdots n_r^{k_r}}\in \mathbb{C},
\end{equation}
which converges if $({\bfk};\bfz)$ is \emph{admissible}, that is, $(k_1,z_1)\ne (1,1)$ (see \cite[Ch. 15]{Zhao2016}). The level two colored MZVs are often called \emph{Euler sums} or \emph{alternating MZVs}. In this case, namely,
when $(z_1,\dotsc,z_r)\in\{\pm 1\}^r$ and $(k_1,z_1)\ne (1,1)$, we denote
$\ze(\bfk;\bfz)= \Li_\bfk(\bfz)$. Further, we put a bar on top of
$k_{j}$ if $z_{j}=-1$. For example,
\begin{equation*}
\ze(\bar2,3,\bar1,4)=\ze(2,3,1,4;-1,1,-1,1).
\end{equation*}
More generally, for any composition $(k_1,\dotsc,k_r)\in\N^r$, the \emph{multiple polylogarithm function} with $r$-variables is defined by
\begin{align}\label{equ:classicalLi}
\Li_{\bfk}(\bfx)\equiv \Li_{k_1,\dotsc,k_r}(x_1,\dotsc,x_r):=\sum_{n_1>\cdots>n_r>0} \frac{x_1^{n_1}\dotsm x_r^{n_r}}{n_1^{k_1}\dotsm n_r^{k_r}}
\end{align}
which converges if $|x_1\cdots x_j|<1$ for all $j=1,\dotsc,r$. They can be analytically continued to a multi-valued meromorphic function on $\mathbb{C}^r$ (see \cite{Zhao2007d}). In particular, if $x_1=x,x_2=\cdots=x_r=1$, then $\Li_{k_1,\ldots,k_r}(x,\{1\}_{r-1})$ is the single-variable multiple polylogarithm function $\Li_{k_1,\ldots,k_r}(x)$.

Kaneko and Tsumura \cite{KanekoTs2018b,KanekoTs2019} introduced the following new kind of multiple polylogarithm functions of level two
\begin{align}\label{b1}
{\rm A}(k_1,\dotsc,k_r;x): &= 2^r\sum\limits_{n_1> \cdots> n_r>0\atop n_j\equiv r+1-j\pmod{2}} {\frac{{{x^{{n_1}}}}}{{n_1^{{k_1}}n_2^{{k_2}} \cdots n_r^{{k_r}}}}}\nonumber\\
&=2^r \sum\limits_{m_1>\cdots>m_r>0}  {\frac{{{x^{{2m_1-r}}}}}{{(2m_1-r)^{{k_1}}(2m_2-r+1)^{{k_2}} \dotsm (2m_r-1)^{{k_r}}}}},
\end{align}
where $|x|\leq 1$, $k_1,\dotsc,k_r$ are positive integers and $(k_1,x)\neq (1,1)$. We call them \emph{Kaneko--Tsumura {\rm A}-functions}. Note that $2^{-r}{\rm A}(k_1,\dotsc,k_r;x)$ is denoted by ${\rm Ath}(k_1,\dotsc,k_r;x)$ in \cite{KanekoTs2018b}. Clearly, if $k_1>1$ and $x=1$,
then the Kaneko--Tsumura {\rm A}-functions in \eqref{b1} reduce to the \emph{multiple $T$-values} defined by
\begin{align}\label{Defn-MTVs}
T(k_1,\dotsc,k_r):&= \sum_{n_1>\cdots>n_r>0 \atop n_j\equiv r+1-j\pmod{2}} \frac{2^r}{n_1^{k_1}n_2^{k_2}\dotsm n_r^{k_r}}\nonumber\\
&=\sum\limits_{m_1>\cdots>m_r>0}  {\frac{2^r}{{(2m_1-r)^{{k_1}}(2m_2-r+1)^{{k_2}} \dotsm (2m_r-1)^{{k_r}}}}}.
\end{align}

Similarly to the multiple harmonic sums and multiple harmonic star sums, for a composition $\bfk=(k_1,\ldots,k_r)$ and positive integer $n$ with $\alpha\in \mathbb{C}\backslash \N_0^-$, we define the \emph{Hurwitz-type multiple harmonic sums} and \emph{Hurwitz-type multiple harmonic star sums} by
\begin{align}
\zeta_n(\bfk;\alpha):=\sum\limits_{n\geq n_1>\cdots>n_r>0 } \frac{1}{(n_1+\alpha-1)^{k_1}\cdots (n_r+\alpha-1)^{k_r}}\label{HTMHSs}
\end{align}
and
\begin{align}
\zeta^\star_n(\bfk;\alpha):=\sum\limits_{n\geq n_1\geq\cdots\geq n_r>0} \frac{1}{(n_1+\alpha-1)^{k_1}\cdots (n_r+\alpha-1)^{k_r}}\label{THMHSSs},
\end{align}
respectively. If $n<r$ then ${\zeta_n}(\bfk;\alpha):=0$ and ${\zeta _n}(\emptyset;\alpha)={\zeta^\star _n}(\emptyset;\alpha):=1$. Further, for $k_1>1$, we let
\begin{align}
\zeta(\bfk;\alpha):=\lim_{n\rightarrow \infty}{\zeta _n}(\bfk;\alpha) \quad\text{and}\quad
{\zeta^\star}( \bfk;\alpha):=\lim_{n\rightarrow \infty}{\zeta^\star_n}( \bfk;\alpha)\label{HTMZVs+HTMZSVs}
\end{align}
and call them \emph{Hurwitz-type multiple zeta values} \cite{MXY2021} and \emph{Hurwitz-type multiple zeta star values}, respectively. Similarly, for $k_1>1$ and $\alpha\in \mathbb{C}\backslash \N_0^-$, we define the \emph{Hurwitz-type multiple $T$-values} by
\begin{align}\label{Defn-HTMTVs}
T(\bfk;\alpha)&\equiv T(k_1,\dotsc,k_r;\alpha)\nonumber\\
:&=\sum\limits_{m_1>\cdots>m_r>0}  {\frac{2^r}{{(2m_1-r-1+\alpha)^{{k_1}}(2m_2-r+\alpha)^{{k_2}} \dotsm (2m_r-2+\alpha)^{{k_r}}}}}.
\end{align}
Clearly, $\zeta(\bfk;1)=\zeta(\bfk),\zeta^\star(\bfk;1)=\zeta^\star(\bfk)$ and $T(\bfk;1)=T(\bfk)$.

More generally, for a composition $\bfk=(k_1,\ldots,k_r)$, positive integer $n$ and $\bfa=(a_1,\ldots,a_r)\in \mathbb{C}^r$ with all $a_j\in \mathbb{C}\backslash \N_0^-$, we define the Hurwitz-type multiple harmonic sums and Hurwitz-type multiple harmonic star sums with $r$-variables by
\begin{align}
\zeta_n(\bfk;\bfa)\equiv\zeta_n(k_1,\ldots,k_r;a_1,\ldots,a_r):=\sum\limits_{n\geq n_1>\cdots>n_r>0 } \frac{1}{(n_1+a_1-1)^{k_1}\cdots (n_r+a_r-1)^{k_r}}\label{HTMHSs-r},
\end{align}
and
\begin{align}
\zeta^\star_n(\bfk;\bfa)\equiv\zeta_n^\star(k_1,\ldots,k_r;a_1,\ldots,a_r):=\sum\limits_{n\geq n_1\geq\cdots\geq n_r>0} \frac{1}{(n_1+a_1-1)^{k_1}\cdots (n_r+a_r-1)^{k_r}}\label{THMHSSs-r},
\end{align}
respectively. If $n<r$ then ${\zeta_n}(\bfk;\bfa):=0$ and ${\zeta _n}(\emptyset;\bfa)={\zeta^\star _n}(\emptyset;\bfa):=1$. Further, for $k_1>1$, let
\begin{align}
\zeta(\bfk;\bfa):=\lim_{n\rightarrow \infty}{\zeta _n}(\bfk;\bfa) \quad\text{and}\quad
{\zeta^\star}( \bfk;\bfa):=\lim_{n\rightarrow \infty}{\zeta^\star_n}( \bfk;\bfa),\label{HTMZVs+HTMZSVs-r}
\end{align}
and call them Hurwitz-type multiple zeta values and Hurwitz-type multiple zeta star values with $r$-variables, respectively. For convenience, if $a_1=\cdots=a_r=a$ in \eqref{HTMHSs-r}-\eqref{HTMZVs+HTMZSVs-r}, let $\zeta_n(\bfk;a):=\zeta_n(\bfk;\{a\}_r),\zeta^\star_n(\bfk;a):=\zeta^\star_n(\bfk;\{a\}_r),\zeta(\bfk;a):=\zeta(\bfk;\{a\}_r)$ and $\zeta^\star(\bfk;a):=\zeta^\star(\bfk;\{a\}_r)$.

In particular, if all $a_j=1/2$ then \eqref{HTMHSs-r} and \eqref{THMHSSs-r} essentially reduce to the
\emph{multiple t-harmonic sum} and \emph{multiple t-harmonic star sum} respectively defined by \cite{XuZhao2021c}
\begin{align*}
t_n(\bfk):=\sum_{n\geq n_1>\cdots>n_r>0} \prod_{j=1}^r \frac{1}{(2n_{j}-1)^{k_{j}}}\quad
\text{and}\quad
t^\star_n(\bfk):=\sum_{n\geq n_1\geq \cdots\geq n_r>0} \prod_{j=1}^r \frac{1}{(2n_{j}-1)^{k_{j}}}.
\end{align*}
If $n<r$ then ${t_n}(\bfk):=0$ and ${t_n}(\emptyset)={t^\star_n}(\emptyset ):=1$. When taking the limit $n\rightarrow \infty$ we get the so-called \emph{multiple $t$-values} and \emph{multiple $t$-star values}, respectively, see \cite{H2019}. Clearly, the multiple $T$-values \cite{KanekoTs2019} and \emph{multiple mixed values}  \cite{XuZhao2020a} are also special cases of the Hurwitz-type MZVs with $r$-variables. For example, letting $a_j=(j+1-r)/2\ (j=1,2,\ldots,r)$ in the first formula of \eqref{HTMZVs+HTMZSVs-r} gives the multiple $T$-value
\begin{align}
&\zeta\left(k_1,k_2,\ldots,k_r;1-\frac{r}{2},1-\frac{r-1}{2},\ldots,1-\frac{1}{2}\right)\nonumber\\
&=2^{k_1+k_2+\cdots+k_r} \sum\limits_{n_1>n_2>\cdots>n_r>0} \frac{1}{(2n_1-r)^{k_1}(2n_2-r+1)^{k_2}\cdots (2n_r-1)^{k_r}}\nonumber\\
&=2^{k_1+k_2+\cdots+k_r-r} T(k_1,k_2,\ldots,k_r).
\end{align}

In this paper, by considering several integrals involving multiple polylogarithm, we will find some explicit relations between the Hurwitz-type multiple zeta values and the following two parametric Ap\'{e}ry-type series
\begin{align}
&\sum_{n=1}^\infty \frac{\ze_{n-1}(k_2,\ldots,k_r)\ze_n^\star(\{1\}_k;1-\alpha)}{n^{k_1+1}\binom{n-\alpha}{n}}\label{Para-Apery-MHS-I},\\
&\sum_{n=1}^\infty \frac{\ze_n(\{1\}_k;\alpha)\ze^\star_n(k_2,\ldots,k_r)}{n^{k_1}} \binom{n+\alpha-1}{n},\label{Para-Apery-MHS-II}
\end{align}
where $(k,k_1,\ldots,k_r)\in \N^{r+1}$, $\Re(\alpha)<1$ with $\alpha \notin \N_0^-$, and the general binomial coefficient $\binom{a}{b}$ is denoted by
\[\binom{a}{b}:=\frac{\Gamma(a+1)}{\Gamma(b+1)\Gamma(a-b+1)}
    \quad\text{for}\ a,b,a-b\notin \N^-.\]
In particular, if $a=n-1/2$ and $b=n\ (n\in \N)$, we have
\[\binom{n-1/2}{n}=\frac{1}{4^n}\binom{2n}{n}.\]
Furthermore, we also study some special cases of the following Ap\'{e}ry-type series involving two parametric binomial coefficients
\begin{align}\label{Para-Apery-MHS-III}
\sum\limits_{n=1}^{\infty}\frac{\zeta_{n}(k_1,\ldots,k_r;\alpha)\zeta_{n}^{\star}(l_1,\ldots,l_p;1-\beta)}{n^{m+2}}
\frac{\binom{n+\alpha-1}{n}}{\binom{n-\beta}{n}}\,,\quad m\in\N_0.
\end{align}
Letting $(\alpha,\beta)\rightarrow (0,0)$ in \eqref{Para-Apery-MHS-III} and applying \eqref{Limit-mhns-para}, we can recover the Kaneko--Yamamoto multiple zeta value $\ze(\bfk\circledast \bfl^\star)$ \cite{KanekoYa2018} defined by
\begin{align}
\zeta({\bfk}\circledast{\bfl}^\star):&=\sum\limits_{0<m_r<\cdots<m_1=n_1\geq \cdots \geq n_p>0} \frac{1}{m_1^{k_1}\cdots m_r^{k_r}n_1^{l_1}\cdots n_p^{l_p}}\nonumber\\
&=\sum\limits_{n=1}^\infty \frac{\zeta_{n-1}(k_2,\ldots,k_r)\zeta^\star_n(l_2,\ldots,l_p)}{n^{k_1+l_1}},
\end{align}
where ${\bfk}=(k_1,\ldots,k_r)$ and ${\bfl}=(l_1,\ldots,l_p)$ are two compositions. It should be emphasized that Kaneko and Yamamoto presented a new ``integral $=$ series" type identity of MZVs, and conjectured that this identity is enough to describe all linear relations of MZVs over $\mathbb{Q}$ (see  \cite[Conj. 4.3]{KanekoYa2018}). It is obvious that, the Arakawa--Kaneko zeta values
\begin{align*}
\xi(p;{\bfk})=\sum\limits_{n=1}^\infty \frac{\zeta_{n-1}(k_2,\ldots,k_r)\zeta^\star_n(\{1\}_{p-1})}{n^{k_1+1}}=\zeta({\bfk}\circledast \{\underbrace{1,\ldots,1}_{p}\}^\star)
\end{align*}
are special cases of the Kaneko--Yamamoto MZVs (see \cite{Ku2010}), where $p,k_1,\ldots,k_r\in \N$.

The main results of the present paper can be used to extend some results on the Ap\'ery-type series studied in \cite{XuZhao2021c}.
See \cite{Akhilesh,Au2020,Au2022,XuZhao2022a,XuZhao2022b} and the references therein for other recent results on Ap\'ery-type series and colored multiple zeta values.

\begin{re} From the asymptotic expansion for the ratio of two gamma functions (see \cite[Sections 2.3 and 2.11]{Luke69.1}), we have
\begin{align*}
\frac{\Gamma(z+a)}{\Gamma(z+b)}=z^{a-b}\left(1+O\Big(\frac1{z}\Big) \right)\quad \text{for } |{\rm arg}(z)|\leq \pi-\varepsilon,\ \varepsilon>0,
    \ |z|\rightarrow \infty.
\end{align*}
Hence, for \eqref{Para-Apery-MHS-III}, by a direct calculation, we obtain
\begin{align*}
\frac{\binom{n+\alpha-1}{n}}{\binom{n-\beta}{n}}=\frac{\Gamma(1-\beta)}{\Gamma(\alpha)} \frac{1}{n^{1-\alpha-\beta}}\left(1+O\Big(\frac1{n}\Big) \right),\quad n \rightarrow \infty.
\end{align*}
Thus, for $m\in\N_0$ and complex numbers $\alpha\in \mathbb{C}\backslash \N_0^-$ and $\beta\in \mathbb{C}\backslash \N$ with $\Re(\alpha+\beta)<2$, the Ap\'{e}ry-type series involving two parametric binomial coefficients in \eqref{Para-Apery-MHS-III} are convergent.
\end{re}

\section{Ap\'ery-type series with one parametric binomial coefficient}\label{sec-mt}

In this section, we will use the iterated integrals to prove some interesting results involving Ap\'ery-type series with one parametric binomial coefficient.

\subsection{Iterated Integrals and Notations}\label{sec-mt}

The theory of iterated integrals was developed firstly by K.T. Chen in the 1960's \cite{KTChen1971,KTChen1977}. It has played important roles in the study of algebraic topology and algebraic geometry in the past half century. Its simplest form is
$$\int_{a}^b f_p(t)dtf_{p-1}(t)dt\cdots f_1(t)dt:=\int_{a<t_p<\cdots<t_1<b}f_p(t_p)f_{p-1}(t_{p-1})\cdots f_1(t_1)dt_1dt_2\cdots dt_p.$$
In particular, according to the definitions, for any (admissible) composition $\bfk:=(k_1,k_2,\ldots,k_r)$ and $\alpha \in \mathbb{C}\setminus \N^-$, we have
\begin{align}
&\Li_{\bfk}(x)=\int_0^x \left(\frac{dt}{1-t}\right)\left(\frac{dt}{t}\right)^{k_r-1}\cdots
\left(\frac{dt}{1-t}\right)\left(\frac{dt}{t}\right)^{k_1-1},\label{Eq-MPL-ItIn}\\
&\A(\bfk;x)=\int_0^x \left(\frac{2dt}{1-t^2}\right)\left(\frac{dt}{t}\right)^{k_r-1}\cdots
\left(\frac{2dt}{1-t^2}\right)\left(\frac{dt}{t}\right)^{k_1-1},\label{Eq-KTA-ItIn}\\
&\ze(\bfk;1+\alpha)=\int_0^1 \frac{t^\alpha dt}{1-t}\left(\frac{dt}{t}\right)^{k_r-1}\left(\frac{dt}{1-t}\right)\left(\frac{dt}{t}\right)^{k_{r-1}-1}\cdots
\left(\frac{dt}{1-t}\right)\left(\frac{dt}{t}\right)^{k_1-1},\label{HTMZVs-Iterated-Integeral}\\
&T(\bfk;1+\alpha)=\int_0^1 \frac{2t^\alpha dt}{1-t^2}\left(\frac{dt}{t}\right)^{k_r-1}\left(\frac{2dt}{1-t^2}\right)\left(\frac{dt}{t}\right)^{k_{r-1}-1}\cdots
\left(\frac{2dt}{1-t^2}\right)\left(\frac{dt}{t}\right)^{k_1-1}.\label{HTMTVs-Iterated-Integeral}
\end{align}

More generally, by the definition of Hurwitz-type MZVs, we have the following iterated integral expressions
\begin{align}
&\zeta(k_1,k_2,\ldots,k_r;a_1+1,a_2+1,\ldots,a_r+1)\nonumber\\
&\quad=\int_0^1 \frac{t^{a_r}dt}{1-t}\left( \frac{dt}{t}\right)^{k_r-1}\frac{t^{a_{r-1}-a_r}dt}{1-t}\left( \frac{dt}{t}\right)^{k_{r-1}-1}\cdots \frac{t^{a_{1}-a_2}dt}{1-t}\left( \frac{dt}{t}\right)^{k_{1}-1}
\end{align}
and
\begin{align}
&\zeta(k_1,k_2,\ldots,k_r;a_1+\cdots+a_r+1,a_2+\cdots+a_r+1,\ldots,a_r+1)\nonumber\\
&\quad=\int_0^1 \frac{t^{a_r}dt}{1-t}\left( \frac{dt}{t}\right)^{k_r-1}\frac{t^{a_{r-1}}dt}{1-t}\left( \frac{dt}{t}\right)^{k_{r-1}-1}\cdots \frac{t^{a_{1}}dt}{1-t}\left( \frac{dt}{t}\right)^{k_{1}-1}.
\end{align}

For convenience, for compositions $\bfk:=(k_1,k_2,\ldots,k_r)$ and $\bfj:=(j_1,j_2,\ldots,j_r)\in \N_0^r$, let $\overleftarrow{\bfk}:=(k_r,\ldots,k_2,k_1)$ be the reversal of $\bfk$ and
\[
B(\bfk;\bfj):=\prod_{i=1}^r\binom{k_i+j_i-1}{j_i}.
\]
Moreover, for two such $\bfk$  and $\bfj$ of the same depth, we denote by $\bfk+\bfj$ the composition obtained by adding the corresponding components.

The Hoffman dual of a composition $\bfk=(k_1,\ldots,k_r)$ is $\bfk^\vee=(k'_1,\ldots,k'_{r'})$ determined by
$|\bfk|:=k_1+\cdots+k_r=k'_1+\cdots+k'_{r'}$ and
\begin{equation*}
\{1,2,\ldots,|\bfk|-1\}
=\Big\{\sum\nolimits_{i=1}^{j} k_i\Big\}_{j=1}^{r-1}
 \coprod \Big\{\sum\nolimits_{i=1}^{j} k_i'\Big\}_{j=1}^{r'-1}.
\end{equation*}
Here $\coprod$ denotes the union of two disjoint sets. Equivalently, $\bfk^\vee$ can be obtained from $\bfk$ by swapping the commas ``,'' and the plus signs ``+'' in the expression
\begin{equation}\label{eq:HdualDef}
 \bfk=(\underbrace{1+\cdots+1}_{\text{$k_1$ times}},\dotsc,\underbrace{1+\cdots+1}_{\text{$k_r$ times}}).
\end{equation}
For example, we have
$({1,1,2,1})^\vee=(3,2)\quad\text{and}\quad ({1,2,1,1})^\vee=(2,3).$ More generally, one obtains
\begin{align}\label{eq:HdualDef2}
{\bfk}^\vee=(\underbrace{1,\ldots,1}_{k_1}+\underbrace{1,\ldots,1}_{k_2}+1,\ldots,1+\underbrace{1,\ldots,1}_{k_r}).
\end{align}
In particular, setting $\bfk_+:=(k_1+1,k_2,\ldots,k_r)$ we have
\[\Big( \overleftarrow{\bfk}^\vee\Big)_+=(1,k_r,k_{r-1},\ldots,k_1)^\vee,
\]
which is equal to the usual dual composition (see for instance \cite[Ch.~5]{Zhao2016} for the precise definition)
of $(k_1+1,k_2,\ldots, k_r)$.

\begin{lem}(\cite[Thm. 2.1]{KWXZ2024})\label{thm-MPL-MAL-Para}
For any composition $\bfk:=(k_1,k_2,\ldots,k_r)$, $k\in \N$, and any $\alpha\in \mathbb{C}\backslash \N$, the following explicit formulas for the Arakawa-Kaneko zeta-type values hold:
\begin{align}
&\int_0^1 \frac{\Li_{\bfk}(x)\log^k(1-x)}{x(1-x)^\alpha}dx\nonumber\\
&\quad=(-1)^kk!\sum_{|\bfj|=k,\dep(\bfj)=n} B\left(\Big( \overleftarrow{\bfk}^\vee\Big)_+;\bfj\right)\ze\left(\Big( \overleftarrow{\bfk}^\vee\Big)_++\bfj;1-\alpha\right),\label{EQ-MPL-LOG-A}\\
&\int_0^1 \frac{\A(\bfk;x)\log^k\left(\frac{1-x}{1+x}\right)}{x\left(\frac{1-x}{1+x}\right)^\alpha}dx\nonumber\\
&\quad=(-1)^kk!\sum_{|\bfj|=k,\dep(\bfj)=n} B\left(\Big( \overleftarrow{\bfk}^\vee\Big)_+;\bfj\right)T\left(\Big( \overleftarrow{\bfk}^\vee\Big)_++\bfj;1-\alpha\right),\label{EQ-MAL-LOG-A}
\end{align}
where $n:=|\bfk|+1-\dep(\bfk)$.
\end{lem}

\subsection{Main Results}

To save space, for any composition $\bfk=(k_1,\dotsc,k_r)\in\N^r$ and $i,j\in\N$, we put
\begin{align*}
&\ora\bfk_{\hskip-2pt i,j}:=
\left\{
  \begin{array}{cl}
    (k_i,\ldots,k_{j}),&\ \hbox{if $i\le j\le r$;} \\
    \emptyset, &\ \hbox{if $i>j$,}
  \end{array}
\right.
 \quad &\ola\bfk_{\hskip-2pt i,j}:=
\left\{
  \begin{array}{cl}
     (k_{j},\ldots,k_i),&\ \hbox{if $i\le j\le r$;} \\
     \emptyset, &\ \hbox{if $i>j$.}
  \end{array}
\right.
\end{align*}
Set $\ora\bfk_{\hskip-2pt i}=\ora\bfk_{\hskip-2pt 1,i}=(k_1,\ldots,k_i)$ and $\ola\bfk_{\hskip-2pt i}=\ola\bfk_{\hskip-2pt i,r}=(k_r,\ldots,k_i)$ for all $1\le i\le r$.

\begin{lem}(cf. \cite[Thm. 2.1]{XuZhao2020b}) Let $r,n\in \N$ and ${\bfk}:=(k_1,\dotsc,k_r)\in \N^r$ with $k_{r+1}:=2$. Then
\begin{align}\label{Eq-x-n-MPL-IN}
 \int_0^1 {x^{n-1}{\rm Li}_{\bfk}(x)dx}
 &=\sum\limits_{j=1}^{k_{1}-1}(-1)^{j-1}\frac{\zeta(k_1+j-1,\overrightarrow{\bfk}_{2,r})}{n^{j}}\nonumber\\&\quad
 +\sum\limits_{l=1}^{r}(-1)^{\overrightarrow{\mid{\bf k}_{l}\mid}-l}\sum\limits_{j=1}^{k_{l+1}-1}(-1)^{j-1}
\frac{\zeta_{n}^{\star}(\overrightarrow{\bfk}_{2,l},j)}{n^{k_{1}}}\zeta(k_{l+1}+1-j,\overrightarrow{\bfk}_{l+2,r}),
\end{align}
where $(k_{r+1}+1-j,\overrightarrow{\bfk}_{r+2,r}):=\emptyset$.
\end{lem}

\begin{lem}(cf. \cite[Thm. 4.1]{Xu2017})\label{lem-S-1}
Define two sequences ${A_m(n)}$ and ${B_m(n)}$ by
$$ A_m(n) = (m-1)!\underset{i = 0}{\overset{m-1}{\sum}}\dfrac{A_i(n)}{i!}\underset{k = 1}{\overset{n}{\sum}}
x_k^{m-i},\quad A_0(n) = 1,\quad \left( {{x_k} \in \mathbb{C},k = 1,2, \cdots ,n} \right),$$
$$ B_m(n) = \underset{k_1 = 1}{\overset{n}{\sum}}x_{k_1}\underset{k_2 = 1}{\overset{k_1}{\sum}}x_{k_2}\cdots\underset{k_m = 1}{\overset{k_{m-1}}{\sum}}x_{k_m},\quad B_0(n) = 1, \quad \left( {{x_k} \in \mathbb{C},k = 1,2, \cdots ,n} \right).$$
Then \[A_m(n) = m!B_m(n).\]
\end{lem}

\begin{lem}(cf. \cite[Thm. 4.2]{Xu2017})\label{lem-S-2}
Define two sequences ${{\bar A}_m(n)}$ and ${{\bar B}_m(n)}$ by
$${\bar A}_m(n) = (m-1)!(-1)^{m-1}\underset{i = 0}{\overset{m-1}{\sum}}(-1)^{i}\dfrac{{\bar A}_i(n)}{i!}\underset{k = 1}{\overset{n}{\sum}}x_k^{m-i},\quad {\bar A}_0(n)=1,$$
$${\bar B}_m(n) = \underset{k_1 = 1}{\overset{n}{\sum}}x_{k_1}\underset{k_2 = 1}{\overset{k_1-1}{\sum}}x_{k_2}\cdots\underset{k_m = 1}{\overset{k_{m-1}-1}{\sum}}x_{k_m},\quad {\bar B}_0(n) = 1.$$
Then \[{\bar A}_m(n) = m!{\bar B}_m(n).\]
\end{lem}

\begin{thm}\label{thm-More-Gen-Fun-1-x-a} For any $k\in\N$ and $\alpha\in \mathbb{C}\backslash \N_0^-$ with $x\in (-1,1)$, we have
\begin{align}\label{eq-More-Gen-Fun-1-x-a}
&\frac{(-1)^{k}}{k!}\frac{\log^{k}(1-x)}{(1-x)^{\alpha}}=\sum\limits_{n=1}^{\infty}\binom{n+\alpha-1}{n}\zeta_{n}(\{1\}_{k};\alpha)x^{n}.
\end{align}
\end{thm}
\begin{proof}
Note that
\begin{align}\label{eq-Gen-Fun-1-x-a}
\frac{1}{(1-x)^{\alpha}}=1+\sum\limits_{n=1}^{\infty}\frac{(\alpha)_{n}}{n!}x^{n}
    =1+\sum\limits_{n=1}^{\infty}\binom{n+\alpha-1}{n}x^{n}\,,\quad x\in(-1,1),
\end{align}
where $(\alpha)_n$ represents the \emph{Pochhammer symbol} (or the \emph{shifted factorial}) given
by
\begin{equation}\label{equ:Pochhammer}
(\alpha)_n=\alpha(\alpha+1)\cdots(\alpha+n-1)\quad \text{and}\quad (\alpha)_0:=1.
\end{equation}
From \cite{Xu2017}, we have
\begin{align}\label{Eq-GF-Log-k}
\log^k(1-x)=(-1)^kk!\sum_{n=1}^\infty \frac{\ze_{n-1}(\{1\}_{k-1})}{n}x^n
    \,,\quad x\in [-1,1).
\end{align}
By the \emph{Cauchy product}, we obtain
\begin{align}\label{eq-coffic-gen-fun}
\frac{(-1)^{k}}{k!}\frac{\log^{k}(1-x)}{(1-x)^{\alpha}}
&=\sum\limits_{n=1}^{\infty}\left\{\frac{\zeta_{n-1}(\{1\}_{k-1})}{n}
    +\sum\limits_{i+j=n\atop i,j\geq1}
        \frac{\zeta_{i-1}(\{1\}_{k-1})}{i}\cdot
        \frac{(\alpha)_{j}}{j!}\right\}\cdot x^{n}\nonumber\\
&=\frac{1}{k!}\frac{\partial^{k}}{\partial\alpha^{k}}\cdot\frac{1}{(1-x)^{\alpha}}
=\frac{1}{k!}\sum\limits_{n=0}^{\infty}\frac{\partial^{k}(\alpha)_{n}}{\partial\alpha^{k}}\frac{x^{n}}{n!}.
\end{align}
Then,
\begin{align}\label{eq-coffic-sum}
\frac{\zeta_{n-1}(\{1\}_{k-1})}{n}+\sum\limits_{i+j=n\atop i,j\geq1}\frac{\zeta_{i-1}(\{1\}_{k-1})}{i}\cdot\frac{(\alpha)_{j}}{j!}
=\frac{1}{k!n!}\frac{\partial^{k}(\alpha)_{n}}{\partial\alpha^{k}}=:a_{n}^{(\alpha)}(k).
\end{align}

By a simple calculation, $\frac{\partial^k(\alpha)_n}{\partial \alpha^k}$ satisfy a recurrence relation in the form
\begin{align}\label{Eq-CF-1}
\frac{{{\partial ^k}{{\left(\alpha\right)}_n}}}{{\partial {\alpha^k}}} =
\sum\limits_{i = 0}^{k-1} \binom{k-1}{i}\frac{\partial^i(\alpha)_n}{\partial \alpha^i}
\left[\psi^{(k -i-1)}(\alpha + n)-\psi^{(k-i-1)}(\alpha) \right]
    \,,\quad\text{for } k\in \N.
\end{align}
Here, $\psi ^{(m)}(x)$ stands for the polygamma function of order $m$ defined by
\[
\psi^{(m)}(x):=\frac{d^m}{dx^m}\psi(x) = \frac{d^{m+1}}{d x^{m+1}}\log \Gamma (x).
\]
Thus
\[{\psi ^{\left( 0 \right)}}(x) = \psi (x) = \frac{{\Gamma '(x)}}{{\Gamma (x)}}\]
holds, where $\psi (x)$ is the digamma function. The functions ${\psi ^{(m)}}(x)$ satisfy the following relations:
\begin{align*}
\psi (x)
    &= -\gamma  + \sum\limits_{k = 0}^\infty
    \left(\frac{1}{k + 1}-\frac{1}{k + x}\right),\\
{\psi ^{\left( n \right)}}(x)
    &= {\left( {-1} \right)^{n + 1}}n!\sum\limits_{k = 0}^\infty  \frac{1}{(x + k)^{n + 1}},\\
\psi \left( {x + n} \right)
    &= \frac{1}{x} + \frac{1}{{x + 1}} +  \cdots  + \frac{1}{{x + n-1}}
    + \psi (x),
\end{align*}
for $x\notin  \N^-_0$ and $n \in \N$. Here, $\gamma$ denotes the \emph{Euler-Mascheroni constant}, defined by
\[\gamma  := \mathop {\lim }\limits_{n \to \infty } \left( {\sum\limits_{k = 1}^n {\frac{1}{k}} -\log n} \right) = -\psi \left( 1 \right) \approx {\rm{ 0 }}{\rm{. 577215664901532860606512 }}....\]

From \eqref{Eq-CF-1} we have
\begin{align*}
\frac{\partial^{k}(\alpha)_{n}}{\partial \alpha^{k}}
&=\sum\limits_{i=0}^{k-1}\binom{k-1}{i}\cdot\frac{\partial^{i}(\alpha)_{n}}{\partial \alpha^{i}}\cdot
[\psi^{(k-i-1)}(\alpha+n)-\psi^{(k-i-1)}(\alpha)]\\
&=\sum\limits_{i=0}^{k-1}\binom{k-1}{i}\cdot\frac{\partial^{i}(\alpha)_{n}}{\partial \alpha^{i}}\cdot(-1)^{k-i}(k-i-1)!
\sum\limits_{m=0}^{\infty}\left\{\frac{1}{(m+\alpha+n)^{k-i}}-\frac{1}{(m+\alpha)^{k-i}}\right\}\\
&=\sum\limits_{i=0}^{k-1}\frac{(k-1)!}{i!(k-i-1)!}\cdot\frac{\partial^{i}(\alpha)_{n}}{\partial \alpha^{i}}\cdot
(-1)^{k-i}(k-i-1)!(-1)\sum\limits_{m=0}^{n-1}\frac{1}{(m+\alpha)^{k-i}}\\
&=(k-1)!(-1)^{k-1}\sum\limits_{i=0}^{k-1}(-1)^{i}\cdot\frac{1}{i!}\cdot\frac{\partial^{i}(\alpha)_{n}}{\partial \alpha^{i}}\cdot\zeta_{n}(k-i;\alpha)\,.
\end{align*}
Hence, we obtain the recurrence relation
\begin{align}\label{recurre-eq-Phs}
\frac{1}{k!}\cdot\frac{\partial^{k}(\alpha)_{n}}{\partial\alpha^{k}}=\frac{(-1)^{k-1}}{k}\sum\limits_{i=0}^{k-1}(-1)^{i}\cdot\frac{1}{i!}
\cdot\frac{\partial^{i}(\alpha)_{n}}{\partial \alpha^{i}}\cdot\zeta_{n}(k-i;\alpha).
\end{align}
In particular, if $k=0$ then $(\alpha)_{n}=\alpha(\alpha+1)\cdots(\alpha+n-1)=\frac{\Gamma(\alpha+n)}{\Gamma(\alpha)}$.

In Lemma \ref{lem-S-2}, letting $x_{k}={1}/(k+\alpha-1)$ yields
\begin{align}\label{recurre-eq-mhss}
\zeta_{n}(\{1\}_{m};\alpha)
    =\frac{(-1)^{m-1}}{m}\sum\limits_{i=0}^{m-1}(-1)^{i}\zeta_{n}(\{1\}_{i};\alpha)\zeta_{n}(m-i;\alpha).
\end{align}
Multiplying \eqref{recurre-eq-mhss} by $(\alpha)_n$, then comparing it with \eqref{recurre-eq-Phs}, we obtain $\frac{\partial^{k}(\alpha)_{n}}{\partial \alpha^{k}}=k!(\alpha)_{n}\zeta_{n}(\{1\}_{k};\alpha)$. Thus, we get
\begin{align}\label{eq-exp-for-coffic-gen-fun}
a_{n}^{(\alpha)}(k)=\frac{k!}{k!n!}(\alpha)_{n}\zeta_{n}(\{1\}_{k};\alpha)
=\frac{(\alpha)_{n}}{n!}\zeta_{n}(\{1\}_{k};\alpha)=\binom{n+\alpha-1}{n}\zeta_{n}(\{1\}_{k};\alpha).
\end{align}
Combining \eqref{eq-coffic-gen-fun}, \eqref{eq-coffic-sum} and \eqref{eq-exp-for-coffic-gen-fun}, we obtain the desired formula.
\end{proof}

\begin{thm}\label{thm-Int-log-aphla-bc} For any $k\in \N_0,n\in \N$ and $\alpha\in \mathbb{C}$ with $\Re(\alpha)<1$, we have
\begin{align}\label{eq-Int-log-aphla-bc}
\int_0^1 {x^{n-1}\frac{\log^{k}(1-x)}{(1-x)^{\alpha}}dx}=(-1)^{k}k!\frac{1}{n\binom{n-\alpha}{n}}\zeta_{n}^{\star}(\{1\}_{k};1-\alpha).
\end{align}
\end{thm}
\begin{proof}
We note that
\begin{equation}\label{equ:betaIntegralExp}
\int_0^1 {x^{n-1}\frac{\log^{k}(1-x)}{(1-x)^{\alpha}}dx}=\frac{\partial^{k}B(a,b)}{\partial b^{k}}\bigg|_{a=n\atop b=1-\alpha},
\end{equation}
where $B(a,b)$ is the \emph{beta function} defined by
\begin{equation}\label{equ:betaFunction}
B\left( {a,b} \right) := \int_0^1 {{x^{a-1}}{{\left( {1-x} \right)}^{b-1}}dx}
 = \frac{{\Gamma \left(a \right)\Gamma \left( b \right)}}{{\Gamma \left( {a+b} \right)}}\,,\quad\text{for }{\mathop{\Re}\nolimits} \left( a\right) > 0,{\mathop{\Re}\nolimits} \left(b\right) > 0.
\end{equation}
By using the definitions of the beta function ${B\left( a,b\right)}$ and the digamma function $\psi (x)$, it is obvious that
\[\frac{{\partial B\left( a,b \right)}}{{\partial a }} = B\left( a,b \right)\left[ {\psi \left( a  \right)-\psi \left( a+b \right)} \right].\]
Therefore, differentiating this equality $k-1$ times, we can deduce that
\begin{align}\label{Eq-Beta-Relat}
\frac{\partial^k B(a,b)}{\partial a^k} = \sum\limits_{i = 0}^{k-1}\binom{k-1}{i}\frac{\partial^i B(a,b)}{\partial a^i} \cdot
\Big[ \psi^{(k-i-1)} (a)-\psi^{(k-i-1)} (a+b) \Big].
\end{align}
Note that $B(a,b)=B(b,a)$. Hence, we have
\begin{align*}
\frac{\partial^{k}}{\partial b^{k}}B(a,b)\bigg|_{a=n\atop b=1-\alpha}
&=\sum\limits_{i=0}^{k-1}\binom{k-1}{i}
\frac{\partial^{i}B(a,b)}{\partial b^{i}}\bigg|_{a=n\atop b=1-\alpha}(-1)^{k-i}(k-i-1)!\\
&\qquad\qquad\qquad\qquad\qquad\times
\left\{\sum\limits_{m=0}^{\infty}\frac{1}{(\alpha+m)^{k-i}}-\sum\limits_{m=0}^{\infty}\frac{1}{(\alpha+n+m)^{k-i}}\right\}\\
&=\sum\limits_{i=0}^{k-1}\binom{k-1}{i}\frac{\partial^{i}B(a,b)}{\partial b^{i}}\bigg|_{a=n\atop b=1-\alpha}
(-1)^{k-i}(k-i-1)!\sum\limits_{m=0}^{n-1}\frac{1}{(\alpha+m)^{k-i}}\\
&=(k-1)!(-1)^{k}\sum\limits_{i=0}^{k-1}\frac{(-1)^{i}}{i!}
    \frac{\partial^{i}B(a,b)}{\partial b^{i}}
    \bigg|_{a=n\atop b=1-\alpha}\zeta_{n}(k-i;\alpha).
\end{align*}
This yields the recurrence relation
\begin{align}\label{eq-recureence-MHS-Beta-Para}
\left\{\frac{(-1)^{k}}{k!}\cdot\frac{\partial^{k}B(a,b)}{\partial b^{k}}\right\}\bigg|_{a=n\atop b=1-\alpha}
=\frac{1}{k}\sum\limits_{i=0}^{k-1}\left\{\frac{(-1)^{i}}{i!}\cdot
\frac{\partial^{i}B(a,b)}{\partial b^{i}}\right\}\bigg|_{a=n\atop b=1-\alpha}\zeta_{n}(k-i;1-\alpha).
\end{align}
In particular, if $k=0$ then
\[B(n,1-\alpha)=\frac{1}{n\binom{n-\alpha}{n}}.\]
On the other hand, taking $x_{k}=\frac{1}{k-\alpha}$ in Lemma \ref{lem-S-1} we see that
\begin{align}\label{eq-recureence-MHSH-Para}
\zeta_{n}^{\star}(\{1\}_{m};1-\alpha)=\frac{1}{m}\sum\limits_{i=0}^{m-1}\zeta_{n}^{\star}(\{1\}_{i};1-\alpha)\zeta_{n}(m-i;1-\alpha).
\end{align}
Multiplying \eqref{eq-recureence-MHSH-Para} by $\frac{1}{n\binom{n-\alpha}{n}}$ and comparing it with \eqref{eq-recureence-MHS-Beta-Para} yield
\begin{align}
\frac{(-1)^{k}}{k!}\frac{\partial^{k}B(a,b)}{\partial b^{k}}\bigg|_{a=n\atop b=1-\alpha}=\frac{1}{n\binom{n-\alpha}{n}}\zeta_{n}^{\star}(\{1\}_{k};1-\alpha).
\end{align}
This completes the proof of Theorem \ref{thm-Int-log-aphla-bc} by \eqref{equ:betaIntegralExp}.
\end{proof}

Obviously, from \eqref{eq-Int-log-aphla-bc} we have
\begin{align}\label{diff-para-binom-mhss}
&\frac{d^k}{d\alpha^k}\frac{1}{\binom{n-\alpha}{n}}=\frac{k!}{\binom{n-\alpha}{n}}\ze^\star_n(\{1\}_k;1-\alpha).
\end{align}
Setting $\alpha=1/2$ in \eqref{eq-Int-log-aphla-bc} yields the well-known result \cite[Eq. (7.2)]{XuZhao2021c}.

\begin{thm}\label{thm-mtss-mhs-cb-aphla} For any $k\in\N_0,\bfk=(k_1,\ldots,k_r)\in \N^r$ and $\alpha\in \mathbb{C}$ with $\Re(\alpha)<1$, we have
\begin{align}\label{eq-mtss-mhs-cb-aphla}
&\sum_{n=1}^\infty \frac{\ze_{n-1}(k_2,\ldots,k_r)\ze_n^\star(\{1\}_k;1-\alpha)}{n^{k_1+1}\binom{n-\alpha}{n}}\nonumber\\
&\quad=\sum_{|\bfj|=k\atop \dep(\bfj)=|\bfk|+1-\dep(\bfk)} B\left(\Big( \overleftarrow{\bfk}^\vee\Big)_+;\bfj\right)\ze\left(\Big( \overleftarrow{\bfk}^\vee\Big)_++\bfj;1-\alpha\right).
\end{align}
\end{thm}
\begin{proof}
Note that
\begin{align*}
\int_0^1 {\frac{{\rm Li}_{k_1,k_2,\ldots,k_r}(x)\log^{k}(1-x)}{x(1-x)^{\alpha}}dx}=\sum_{n=1}^\infty \frac{\ze_{n-1}(k_2,\ldots,k_r)}{n^{k_1}}\int_0^1{x^{n-1}\frac{\log^{k}(1-x)}{(1-x)^{\alpha}}dx}.
\end{align*}
By \eqref{eq-Int-log-aphla-bc}, there holds
\begin{align*}
\int_0^1 {\frac{{\rm Li}_{k_1,k_2,\ldots,k_r}(x)\log^{k}(1-x)}{x(1-x)^{\alpha}}dx}=(-1)^kk!\sum_{n=1}^\infty \frac{\ze_{n-1}(k_2,\ldots,k_r)\ze^\star(\{1\}_k;1-\alpha)}{n^{k_1+1}\binom{n-\alpha}{n}}.
\end{align*}
Thus we deduce the desired result by \eqref{EQ-MPL-LOG-A}.
\end{proof}

For example, we have
\begin{align*}
&\sum\limits_{n=1}^{\infty}\frac{\zeta_{n}^{\star}(1;1-\alpha)}{n^{3}\binom{n-\alpha}{n}}=2\zeta(3,1;1-\alpha)+\zeta(2,2;1-\alpha),\\
&\sum\limits_{n=1}^{\infty}\frac{\zeta_{n-1}(1)\zeta_{n}^{\star}(1;1-\alpha)}
 {n^{3}\binom{n-\alpha}{n}}=3\ze(4,1;1-\alpha)+\ze(3,2;1-\alpha),\\
&\sum\limits_{n=1}^{\infty}\frac{\zeta_{n-1}(2)\zeta_{n}^{\star}(1;1-\alpha)}
 {n^{2}\binom{n-\alpha}{n}}=2\ze(3,2;1-\alpha)+2\ze(2,3;1-\alpha),\\
&\sum\limits_{n=1}^{\infty}\frac{\zeta_{n-1}(2)\zeta_{n}^{\star}(1;1-\alpha)}
 {n^{3}\binom{n-\alpha}{n}}=2\zeta(3,2,1;1-\alpha)+2\zeta(2,3,1;1-\alpha)+\zeta(2,2,2;1-\alpha).
\end{align*}
In particular, letting $k=0$ and $\alpha=0$ in \eqref{eq-mtss-mhs-cb-aphla} gives the duality relation
\[\ze(k_1+1,k_2,\ldots,k_r)=\ze((1,k_r,k_{r-1},\ldots,k_1)^\vee).\]

\begin{re} It should be pointed out that \eqref{eq-mtss-mhs-cb-aphla} was
firstly proved by Chen \cite[Thm. 4.4]{ChenKW19}.
\end{re}

\begin{thm}\label{thm-mtss-mhs-mhss-cb-aphla} For any composition $\bfk:=(k_1,k_2,\ldots,k_r)$, $k\in \N_0$, and any $\alpha\in \mathbb{C}\backslash \Z$, we have
\begin{align}\label{eq-mtss-mhs-mhss-cb-aphla}
\delta_{0,k}\ze(k_1+1,k_2,\ldots,k_r)
&+\sum\limits_{j=1}^{k_{1}-1}(-1)^{j-1}\zeta(k_1+1-j,k_2,\ldots,k_r)
\sum\limits_{n=1}^{\infty}\frac{\binom{n+\alpha-1}{n}\zeta_{n}(\{1\}_{k};\alpha)}{n^{j}}\nonumber\\
&+\sum\limits_{l=1}^{r}(-1)^{\overrightarrow{\mid{\bfk}_{l}\mid}-l}\sum\limits_{j=1}^{k_{l+1}-1}(-1)^{j-1}
\zeta(k_{l+1}+1-j,k_l+2,\ldots,k_r)\nonumber\\
&\quad\quad\quad\quad\quad\quad\quad\times\sum\limits_{n=1}^{\infty}\frac{\binom{n+\alpha-1}{n}\zeta_{n}(\{1\}_{k};\alpha)\zeta_{n}^{\star}(k_2,\ldots,k_l,j)}{n^{k_{1}}}\nonumber\\
&=\sum_{|\bfj|=k\atop \dep(\bfj)=|\bfk|+1-\dep(\bfk)} B\left(\Big( \overleftarrow{\bfk}^\vee\Big)_+;\bfj\right)\ze\left(\Big( \overleftarrow{\bfk}^\vee\Big)_++\bfj;1-\alpha\right),
\end{align}
where $\delta_{0,0}:=1$ and $\delta_{0,k}:=0$ for $k>0$.
\end{thm}
\begin{proof} From \eqref{eq-More-Gen-Fun-1-x-a} and \eqref{eq-Gen-Fun-1-x-a},
we can find that if $k=0$ then
\[
\int_0^1 {\frac{{\rm Li}_{k_1,k_2,\ldots,k_r}(x)}{x(1-x)^{\alpha}}dx}
 =\ze(k_1+1,k_2,\ldots,k_r)+\sum\limits_{n=1}^{\infty}\binom{n+\alpha-1}{n}\int_0^1 {x^{n-1}{\rm Li}_{k_1,k_2,\ldots,k_r}(x)dx},
\]
and if $k\in \N$ then
\begin{align*}
&\int_0^1 {\frac{{\rm Li}_{k_1,k_2,\ldots,k_r}(x)\log^{k}(1-x)}{x(1-x)^{\alpha}}dx}\\
&\quad=(-1)^{k}k!\sum\limits_{n=1}^{\infty}\binom{n+\alpha-1}{n}\zeta_{n}(\{1\}_{k};\alpha)\int_0^1 {x^{n-1}{\rm Li}_{k_1,k_2,\ldots,k_r}(x)dx}.
\end{align*}
The theorem now follows from \eqref{Eq-x-n-MPL-IN} easily.
\end{proof}

Clearly, letting $\alpha=1/2$ in Theorems \ref{thm-mtss-mhs-cb-aphla} and \ref{thm-mtss-mhs-mhss-cb-aphla}, we obtain \cite[Thm. 8.3]{XuZhao2021c}.

\begin{thm}\label{thm-binom-para-diff} For $k\in \N,m\in\N_0$ and $\alpha\in \mathbb{C}\backslash \Z$, we have
\begin{align}
&\sum_{n=1}^\infty \frac{\binom{n+\alpha-1}{n}}{n^{m+1}}=\alpha\sum_{n=1}^\infty \frac{\ze_{n-1}(\{1\}_m)}{n(n-\alpha)},\label{eq-not-mhs-1}\\
&\sum_{n=1}^\infty \frac{\binom{n+\alpha-1}{n}\ze_{n}(\{1\}_k;\alpha)}{n^{m+1}}=\sum_{n=1}^\infty \frac{\ze_{n-1}(\{1\}_m)}{(n-\alpha)^{k+1}}.\label{eq-have-mhs-1}
\end{align}
\end{thm}
\begin{proof}
Consider the integral
\begin{align*}
&\int_0^1 \frac{\log^m(x)}{x}\left(\frac1{(1-x)^\alpha}-1 \right)=\sum_{n=1}^\infty \binom{n+\alpha-1}{n}\int_0^1 x^{n-1}\log^m(x)dx\\
&\quad=(-1)^mm!\sum_{n=1}^\infty \frac{\binom{n+\alpha-1}{n}}{n^{m+1}}
    =\int_0^1 \frac{\log^m(1-x)}{1-x}\left(\frac1{x^\alpha}-1 \right)dx\\
&\quad=(-1)^mm!\sum_{n=1}^\infty \ze_{n-1}(\{1\}_m) \int_0^1 x^{n-1}\left(\frac1{x^\alpha}-1 \right)dx
    =(-1)^mm!\alpha\sum_{n=1}^\infty \frac{\ze_{n-1}(\{1\}_m)}{n(n-\alpha)}.
\end{align*}
This immediately leads to \eqref{eq-not-mhs-1}. Noting that from \eqref{eq-More-Gen-Fun-1-x-a} we have
\begin{align}\label{eq-diff-pcb}
\frac{d^k}{d\alpha^k}\binom{n+\alpha-1}{n}=k!\binom{n+\alpha-1}{n}\ze_n(\{1\}_k;\alpha).
\end{align}
Thus,  differentiating \eqref{eq-not-mhs-1} $k-1$ times with respect to $\alpha$ yields \eqref{eq-have-mhs-1} easily.
\end{proof}

Some explicit evaluations of the series on the right-hand side of \eqref{eq-not-mhs-1} and \eqref{eq-have-mhs-1} can be found in \cite{Xu2017-JMAA}. For example,
from \cite[Thm. 2.6]{Xu2017-JMAA}, we deduce
\begin{align}
\sum\limits_{n = 1}^\infty  {\frac{{{H_n}}}
{{\left( {n + a} \right)\left( {n + b} \right)}}}  &= \frac{1}
{{b-a}}\mathop {\lim }\limits_{x \to 1} \left\{ {\sum\limits_{n = 1}^\infty  {\left( {\frac{{{x^n}}}
{{n + a}}-\frac{{{x^n}}}
{{n + b}}} \right)\left( {\sum\limits_{j = 1}^n {\frac{{{x^{n-j}}}}
{j}} } \right)} } \right\}\nonumber\\
&= \sum\limits_{n = 1}^\infty  {\frac{1}
{{n\left( {n + a} \right)\left( {n + b} \right)}}}  + \frac{1}
{{2\left( {b-a} \right)}}\sum\limits_{n = 1}^\infty  {\left\{ {\frac{1}
{{{{\left( {n + a} \right)}^2}}}-\frac{1}
{{{{\left( {n + b} \right)}^2}}}} \right\}}\nonumber \\
& \quad+ \frac{1}
{2}\left( {\sum\limits_{n = 1}^\infty  {\frac{1}
{{\left( {n + a} \right)\left( {n + b} \right)}}} } \right)\left( {a\sum\limits_{n = 1}^\infty  {\frac{1}
{{n\left( {n + a} \right)}} + b\sum\limits_{n = 1}^\infty  {\frac{1}
{{n\left( {n + b} \right)}}} } } \right),
\end{align}
where $a,b\in \mathbb{C}\backslash \N^-$ with $a\neq b$, and $H_n=\ze_n(1)$ are the classical harmonic numbers.
Letting $a=\alpha$ and $b=0$ yields
\begin{align}\label{case-Harmonic-N}
\sum_{n=1}^\infty \frac{\ze_{n-1}(1)}{n(n+\alpha)}=\frac{\ze(2)-\ze(2;1+\alpha)}{2\alpha}+\frac{(\psi(1+\alpha)+\gamma)^2}{2\alpha}.
\end{align}
Moreover, from \eqref{case-Harmonic-N} and \cite[Thm. 3.1 and Eq. (2.24)]{Xu2017-JMAA}, we further find that for $k\in \N$ and $\alpha\in \mathbb{C}\backslash \N^-$, the parametric Euler sums
\[\sum_{n=1}^\infty \frac{\ze_{n-1}(1)}{(n+\alpha)^{k+1}}\quad \text{and}\quad \sum_{n=1}^\infty \frac{\ze_{n-1}(1,1)}{(n+\alpha)^{k+1}}\]
can be expressed in terms of combinations of products of Hurwitz zeta values and digamma function. In fact, we have the following general result.

\begin{thm} For any positive integers $k,m$ and complex $\Re(\alpha)>0$, the parametric Euler sums
\[\sum_{n=1}^\infty \frac{\ze_{n-1}(\{1\}_m)}{n(n+\alpha)}\quad \text{and}\quad \sum_{n=1}^\infty \frac{\ze_{n-1}(\{1\}_m)}{(n+\alpha)^{k+1}}\]
can be evaluated by Hurwitz zeta values and digamma function.
\end{thm}
\begin{proof} Changing $k$ by $m+1$ in \eqref{Eq-GF-Log-k} and differentiating, one obtains
\begin{align*}
\sum_{n=1}^\infty \ze_n(\{1\}_m)x^n = \frac{(-1)^m}{m!} \frac{\log^m(1-x)}{1-x}.
\end{align*}
Multiplying it by $(1-x^\alpha)$ and integrating over $(0,1)$ yields
\begin{align*}
&\alpha\sum_{n=1}^\infty \frac{\ze_n(\{1\}_m)}{(n+1)(n+\alpha+1)}=\alpha\sum_{n=1}^\infty \frac{\ze_{n-1}(\{1\}_m)}{n(n+\alpha)}\\
&= \frac{(-1)^m}{m!} \int_0^1 \frac{\log^m(1-x)}{1-x}(1-x^\alpha)dx= \frac{(-1)^{m+1}}{(m+1)!} \int_0^1 (1-x^\alpha)d\log^{m+1}(1-x)\\
&=\frac{(-1)^{m+1}}{(m+1)!} \alpha \int_0^1 x^{\alpha-1}\log^{m+1}(1-x)dx.
\end{align*}
Hence, we have
\begin{align*}
\sum_{n=1}^\infty \frac{\ze_{n-1}(\{1\}_m)}{n(n+\alpha)}=\frac{(-1)^{m+1}}{(m+1)!}\int_0^1 x^{\alpha-1}\log^{m+1}(1-x)dx.
\end{align*}
Noting the fact that
\begin{align*}
\alpha\sum_{n=1}^\infty \frac{\ze_{n-1}(\{1\}_m)}{n(n+\alpha)}=\sum_{n=1}^\infty \ze_{n-1}(\{1\}_m)\left(\frac1{n}-\frac1{n+\alpha}\right)
\end{align*}
and differentiating this equality $k$ times, we can deduce that
\begin{align*}
\sum_{n=1}^\infty \frac{\ze_{n-1}(\{1\}_m)}{(n+\alpha)^{k+1}}=\frac{(-1)^{m+k}}{(m+1)!k!} \frac{d^k}{d\alpha^k} \left\{\alpha \int_0^1 x^{\alpha-1}\log^{m+1}(1-x)dx\right\}.
\end{align*}
On the other hand, from \cite[Thm. 2.1]{Xu2017-2}, we know the result that the integral on the right hand of above
\[\int_0^1 x^{\alpha-1}\log^{m+1}(1-x)dx\]
can be expressed in terms of Hurwitz zeta values and digamma function. Thus, combining related results, we obtain the statement of this theorem.
\end{proof}

From Theorems \ref{thm-mtss-mhs-mhss-cb-aphla} and \ref{thm-binom-para-diff}
we can calculate the following cases with the help of \eqref{case-Harmonic-N}:
\begin{align*}
&\sum_{n=1}^\infty \frac{\binom{n+\alpha-1}{n}}{n}=-(\psi(1-\alpha)+\gamma),\\
&\sum\limits_{n=1}^{\infty}\frac{\binom{n+\alpha-1}{n}\zeta_{n}(\{1\}_k;\alpha)}{n}=\zeta(k+1;1-\alpha),\quad (k\geq 1)\\
&\sum_{n=1}^\infty \frac{\binom{n+\alpha-1}{n}}{n^2}=\frac 1{2}(\ze(2;1-\alpha)-\ze(2))-\frac 1{2}(\psi(1-\alpha)+\gamma)^2,\\
&\sum\limits_{n=1}^{\infty}\frac{\binom{n+\alpha-1}{n}\zeta_{n}^{\star}(\{1\}_{k})}{n^{2}}
 =\zeta(k+1,1)-\zeta(k+1,1;1-\alpha)-\zeta(k+1)(\psi(1-\alpha)+\gamma)\quad (k\geq 1),\\
&\sum\limits_{n=1}^{\infty}\frac{\binom{n+\alpha-1}{n}\zeta_{n}(1;\alpha)\zeta_{n}^{\star}(1)}{n^{2}}
=\zeta(2)\zeta(2;1-\alpha)-2\zeta(3,1;1-\alpha)-\zeta(2,2;1-\alpha),\\
&\sum\limits_{n=1}^{\infty}\frac{\binom{n+\alpha-1}{n}\zeta_{n}(1;\alpha)\zeta_{n}^{\star}(1,1)}{n^{2}}
=\ze(3)\ze(2;1-\alpha)-\ze(3,2;1-\alpha)-3\ze(4,1;1-\alpha),\\
&\sum\limits_{n=1}^{\infty}\frac{\binom{n+\alpha-1}{n}\zeta_{n}(1;\alpha)\zeta_{n}^{\star}(2,1)}{n^{2}}=2\zeta(3,2,1;1-\alpha)+2\zeta(2,3,1;1-\alpha)+\zeta(2,2,2;1-\alpha)\\
&\quad\quad\quad\quad\quad\quad\quad\quad\quad\quad\quad\quad\quad+\frac{7}{4}\zeta(4)\zeta(2;1-\alpha)-2\zeta(2)\zeta(3,1;1-\alpha)-\zeta(2)\zeta(2,2;1-\alpha).
\end{align*}

\begin{re} It is obvious that if $\alpha\in \mathbb{C}\backslash \N_0^-$ and $\Re(\alpha)<1$, then the Ap\'ery-type series on the left-hand side of \eqref{eq-mtss-mhs-cb-aphla} and the linear combination of Ap\'ery-type series on the left-hand side of \eqref{eq-mtss-mhs-mhss-cb-aphla} are equal. In particular, by considering the integral
\[\int_0^1 \frac{\log^{r+k}(1-x)}{x(1-x)^\alpha}dx,\quad\text{for } r,k\in \N,\]
we can deduce the following duality formula
\begin{align}
\sum_{n=1}^\infty \binom{n+\alpha-1}{n}\frac{\ze_n(\{1\}_k;\alpha)\ze_n^\star(\{1\}_r)}{n}&=\sum_{n=1}^\infty \frac{\ze_{n-1}(\{1\}_{r-1})\ze^\star_n(\{1\}_k;1-\alpha)}{n^2\binom{n-\alpha}{n}}\nonumber\\
&=\binom{k+r}{k}\ze(k+r+1;1-\alpha).
\end{align}
\end{re}

\subsection{Several general results}

In \eqref{EQ-MPL-LOG-A}, letting $\Big( \overleftarrow{\bfk}^\vee\Big)_+=(m_1,\ldots,m_p)$, we obtain $(k_1,\ldots,k_r)=(m_p,\ldots,m_2,m_1-1)^\vee$, where $m_1>1,m_j\geq 1$ for $j=2,3,\ldots,p$. Then, \eqref{EQ-MPL-LOG-A} can be written in the following form
\begin{align}\label{EQ-MPL-LOG-A-change}
&\int_0^1 \frac{\Li_{(m_p,\ldots,m_2,m_1-1)^\vee}(x)\log^k(1-x)}{x(1-x)^\alpha}dx=(-1)^k\frac{d^k}{d\alpha^k}\ze(m_1,m_2,\ldots,m_p;1-\alpha)\nonumber\\
&=(-1)^kk!\sum_{i_1+i_2+\cdots+i_p=k\atop i_j\geq 0,\ \forall j}\left\{\prod\limits_{j=1}^p\binom{m_j+i_j-1}{i_j}\right\}\ze(m_1+i_1,m_2+i_2,\ldots,m_p+i_p;1-\alpha).
\end{align}
In particular, if $m_1=\cdots=m_p=m\geq 2$ then $(\{m\}_{p-1},m-1)^\vee=(\{1\}_{m-1},\{2,\{1\}_{m-2}\}_{p-1})$.

\begin{lem}(cf. \cite[Thm. 3.5]{X2020})\label{lem-General-HTMZV-HZV} For any $k,m,p\in\N$ and $a\in \mathbb{C}\backslash \N^-$, we have
\begin{align}\label{eq-differ-k-HMZVs-HZV}
&\frac{(-1)^k}{k!}\frac{\partial^k}{\partial a^k} \left\{\zeta(\{m+1\}_{p};a+1)\right\} =\sum\limits_{c_1+2c_2+\cdots+pc_p=p\atop c_1,c_2,\ldots,c_p \geq 0} \left\{\prod\limits_{j=1}^p \frac{(-1)^{(j-1)c_j}}{c_j!j^{c_j}} \right\}\nonumber\\
& \quad \quad \quad \quad \quad \times \sum\limits_{k_1+k_2+\cdots+k_{{|\bf c|}_p}=k\atop k_1,k_2,\ldots,k_{{|\bf c|}_p}\geq 0} \prod_{i=1}^p \prod\limits_{j_i={|\bf c|}_{i-1}+1}^{{|\bf c|}_i} \binom{im+i-1+k_{j_i}}{k_{j_i}} \zeta(im+i+k_{j_i};a+1),
\end{align}
where $\prod_{j=1}^0(\cdot):=1$, $|{\bf c}|_0:=0$ and $|{\bf c}|_i:=c_1+c_2+\cdots+c_i$ for all $i\ge 1$.
\end{lem}

\begin{thm}\label{thm-MHSs-MSHSs-para-HZVs}
For any $k\in \N_0,p\in \N$ and $2\leq m\in \N$, and $\alpha\in \mathbb{C}$ with $\Re(\alpha)<1$, we have
\begin{align}\label{EQ-MHSs-MSHSs-para-HZVs}
&\sum_{n=1}^\infty \frac{\ze_{n-1}(\{1\}_{m-2},\{2,\{1\}_{m-2}\}_{p-1})\ze_n^\star(\{1\}_k;1-\alpha)}{n^2\binom{n-\alpha}{n}}\nonumber\\
&\quad=\sum\limits_{c_1+2c_2+\cdots+pc_p=p\atop c_1,c_2,\ldots,c_p \geq 0} \left\{\prod\limits_{j=1}^p \frac{(-1)^{(j-1)c_j}}{c_j!j^{c_j}} \right\}\nonumber\\
&\quad\quad\times\sum\limits_{k_1+k_2+\cdots+k_{{|\bf c|}_p}=k\atop k_1,k_2,\ldots,k_{{|\bf c|}_p}\geq 0} \prod_{i=1}^p \prod\limits_{j_i={|\bf c|}_{i-1}+1}^{{|\bf c|}_i} \binom{im-1+k_{j_i}}{k_{j_i}} \zeta(im+k_{j_i};1-\alpha).
\end{align}
\end{thm}
\begin{proof}
This theorem follows immediately from \eqref{eq-mtss-mhs-cb-aphla}, \eqref{EQ-MPL-LOG-A-change} and \eqref{eq-differ-k-HMZVs-HZV} with $a=1-\alpha$ and $m+1$ replaced by $m$.
\end{proof}

Setting $p=1,2$ in \eqref{EQ-MHSs-MSHSs-para-HZVs}, we obtain
\begin{align*}
&\sum_{n=1}^\infty \frac{\ze_{n-1}(\{1\}_{m-2})\ze^\star_n(\{1\}_k;1-\alpha)}{n^2\binom{n-\alpha}{n}}=\binom{m+k-1}{k}\ze(m+k;1-\alpha),\\
&\sum_{n=1}^\infty \frac{\ze_{n-1}(\{1\}_{m-2},2,\{1\}_{m-2})\ze^\star_n(\{1\}_k;1-\alpha)}{n^2\binom{n-\alpha}{n}}
=-\frac{1}{2}\binom{2m+k-1}{k}\zeta(2m+k;1-\alpha)\\
&\qquad +\frac{1}{2}\sum\limits_{k_1+k_2=k\atop k_1,k_2\geq 0}\binom{m+k_1-1}{k_1}\binom{m+k_2-1}{k_2} \zeta(m+k_1;1-\alpha)\zeta(m+k_2;1-\alpha) .
\end{align*}

\begin{thm}\label{thm-mtss-mhs-mhss-cb-aphla-ex} For $k\in \N_0,p\in \N$ and $2\leq m\in \N$ with $\alpha\in \mathbb{C}\backslash \Z$, we have
\begin{align}\label{eq-mtss-mhs-mhss-cb-aphla-ex}
&\delta_{0,k}\ze(\{m\}_p)+\sum_{l=1}^p (-1)^{l-1}\ze(\{m\}_{p-l})\sum_{n=1}^\infty \binom{n+\alpha-1}{n}\frac{\ze_n(\{1\}_k;\alpha)\ze_n^\star(\{\{1\}_{m-2},2\}_{l-1},\{1\}_{m-1})}{n}\nonumber\\
&\quad=\sum\limits_{c_1+2c_2+\cdots+pc_p=l\atop c_1,c_2,\ldots,c_p \geq 0} \left\{\prod\limits_{j=1}^p \frac{(-1)^{(j-1)c_j}}{c_j!j^{c_j}} \right\}\nonumber\\
&\quad\quad\times\sum\limits_{k_1+k_2+\cdots+k_{{|\bf c|}_p}=k\atop k_1,k_2,\ldots,k_{{|\bf c|}_p}\geq 0} \prod_{i=1}^p \prod\limits_{j_i={|\bf c|}_{i-1}+1}^{{|\bf c|}_i} \binom{im-1+k_{j_i}}{k_{j_i}} \zeta(im+k_{j_i};1-\alpha).
\end{align}
\end{thm}
\begin{proof}
By the proof of Theorem \ref{thm-mtss-mhs-mhss-cb-aphla}, we know that
\begin{align*}
 &\int_0^1 {\frac{{\rm Li}_{k_1,k_2,\ldots,k_r}(x)\log^{k}(1-x)}{x(1-x)^{\alpha}}dx}=(-1)^{k}k!\delta_{0,k}\ze(k_1+1,k_2,\ldots,k_r)\\&\quad\quad\quad\quad\quad\quad\quad\quad+(-1)^{k}k!\sum\limits_{n=1}^{\infty}\binom{n+\alpha-1}{n}\zeta_{n}(\{1\}_{k};\alpha)\int_0^1 {x^{n-1}{\rm Li}_{k_1,k_2,\ldots,k_r}(x)dx}.
\end{align*}
Setting $(k_1,\ldots,k_r)=(\{1\}_{m-1},\{2,\{1\}_{m-2}\}_{p-1})$ in \eqref{eq-mtss-mhs-mhss-cb-aphla} and applying \eqref{EQ-MPL-LOG-A-change} yield
\begin{align*}
\delta_{0,k}\ze(\{m\}_p)&+\sum_{l=1}^p (-1)^{l-1}\ze(\{m\}_{p-l})\sum_{n=1}^\infty \binom{n+\alpha-1}{n}\frac{\ze_n(\{1\}_k;\alpha)\ze_n^\star(\{\{1\}_{m-2},2\}_{l-1},\{1\}_{m-1})}{n}\\
&=\frac1{k!} \frac{d^k}{d\alpha^k}\ze(\{m\}_p;1-\alpha),
\end{align*}
where we have used the duality relation
\[\ze(\{2,\{1\}_{m-2}\}_p)=\ze(\{m\}_p).\]
Finally, applying \eqref{eq-differ-k-HMZVs-HZV} yields the desired evaluation.
\end{proof}

Setting $p=1$ and $2$ in \eqref{eq-mtss-mhs-mhss-cb-aphla-ex}, we obtain
\begin{align}
&\sum_{n=1}^\infty \binom{n+\alpha-1}{n}\frac{\ze_n(\{1\}_k;\alpha)\ze_n^\star(\{1\}_{m-1})}{n}=\binom{m+k-1}{k}\ze(m+k;1-\alpha)-\delta_{0,k}\ze(m),\label{for-case1}\\
&\sum_{n=1}^\infty \binom{n+\alpha-1}{n}\frac{\ze_n(\{1\}_k;\alpha)\ze_n^\star(\{1\}_{m-2},2,\{1\}_{m-1})}{n}\nonumber\\
&=-\frac{1}{2}\sum\limits_{k_1+k_2=k\atop k_1,k_2\geq 0}\binom{m+k_1-1}{k_1}\binom{m+k_2-1}{k_2} \zeta(m+k_1;1-\alpha)\zeta(m+k_2;1-\alpha)+\delta_{0,k}\ze(m,m)\nonumber\\&\quad+\frac{1}{2}\binom{2m+k-1}{k}\zeta(2m+k;1-\alpha)+\binom{m+k-1}{k}\ze(m)\ze(m+k;1-\alpha).\label{for-case2}
\end{align}
In particular, letting $\alpha\rightarrow 0$ in \eqref{for-case2} and applying \eqref{Limit-mhns-para} we recover \cite[Eq. (3.24)]{X2020}.

Clearly, from Theorem \ref{thm-mtss-mhs-mhss-cb-aphla-ex} we can conclude that the parametric Ap\'ery-type series
\[\sum_{n=1}^\infty \binom{n+\alpha-1}{n}\frac{\ze_n(\{1\}_k;\alpha)\ze_n^\star(\{\{1\}_{m-2},2\}_{p-1},\{1\}_{m-1})}{n}\]
are expressible in terms of products of Hurwitz zeta values.

In fact, we can prove the following more general theorem.
\begin{thm}\label{thm-mzvs-cb-mhs-mhss-hmzv}
For any $k\in \N_0,\alpha\in \mathbb{C}\backslash \Z$ and $\bfm=(m_1,m_2,\ldots,m_p)\in\N^p$ with $m_1,m_p\geq 2$,
\begin{align}\label{eq-mzvs-cb-mhs-mhss-hmzv}
&\sum_{j=1}^p (-1)^{j-1} \ze(m_p,\ldots,m_{j+1})\sum_{n=1}^\infty \binom{n+\alpha-1}{n}\frac{\ze_n(\{1\}_k;\alpha)\ze_n^\star((m_1-1,m_2,\ldots,m_j)^\vee)}{n}\nonumber\\
&=\frac1{k!}\frac{d^k}{d\alpha^k}\ze(m_p,\ldots,m_1;1-\alpha)-\delta_{0,k}\ze(m_p,\ldots,m_1)\nonumber\\
&=\sum_{i_1+\cdots+i_p=k\atop i_j\geq 0,\ \forall j}\left\{\prod\limits_{j=1}^p\binom{m_j+i_j-1}{i_j}\right\}\ze(m_p+i_p,\ldots,m_1+i_1;1-\alpha)-\delta_{0,k}\ze(m_p,\ldots,m_1).
\end{align}
\end{thm}
\begin{proof}
Replacing $m_1$ by $m_1-1$ in \cite[Eq. (2.8)]{XuZhao2021a} and setting $x=0$, we get the iterated integral expression
\begin{align}\label{IMP-2}
&\sum\limits_{j=1}^p (-1)^{j-1}\ze(m_p,\ldots,m_{j+1})\zeta^\star_n((m_1-1,m_2,\ldots,m_j)^\vee)\nonumber\\
&\quad=n\int_0^1
\left(\frac{dt}{1-t}\right)^{m_p-1}\frac{dt}{t} \cdots \left(\frac{dt}{1-t}\right)^{m_2-1}\frac{dt}{t}\left(\frac{dt}{1-t}\right)^{m_1-1}t^{n-1}dt.
\end{align}
Multiplying \eqref{IMP-2} by $\binom{n+\alpha-1}{n}/n$, summing up, and then applying \eqref{eq-Gen-Fun-1-x-a}, we obtain
the following relation by the definition of Hurwitz-type MZVs
\begin{align*}
&\sum_{j=1}^p (-1)^{j-1} \ze(m_p,\ldots,m_{j+1})\sum_{n=1}^\infty \binom{n+\alpha-1}{n}\frac{\ze_n^\star((m_1-1,m_2,\ldots,m_j)^\vee)}{n}\\
&\quad=\int_0^1
\left(\frac{dt}{1-t}\right)^{m_p-1}\frac{dt}{t} \cdots \left(\frac{dt}{1-t}\right)^{m_2-1}\frac{dt}{t}\left(\frac{dt}{1-t}\right)^{m_1-1}\left(\frac{dt}{t(1-t)^\alpha}-\frac{dt}{t}\right)dt\\
&\quad=\int_0^1 \left(\frac{t^{-\alpha}dt}{1-t}-\frac{dt}{1-t}\right)\left(\frac{dt}{t}\right)^{m_1-1}\frac{dt}{1-t}\left(\frac{dt}{t}\right)^{m_2-1}\cdots\frac{dt}{1-t}\left(\frac{dt}{t}\right)^{m_p-1}\\
&\quad=\ze(m_p,\ldots,m_1;1-\alpha)-\ze(m_p,\ldots,m_1).
\end{align*}
Now, differentiating the above formula $k$ times with respect to $\alpha$ and using \eqref{eq-diff-pcb},
we can finally deduce the desired result.
\end{proof}

Clearly, letting $m_1=\cdots=m_p=m\geq 2$ in Theorem \ref{thm-mzvs-cb-mhs-mhss-hmzv} yields Theorem \ref{thm-mtss-mhs-mhss-cb-aphla-ex}. In particular, from \eqref{eq-More-Gen-Fun-1-x-a} and \eqref{Eq-GF-Log-k}, we have
\begin{align}\label{Limit-mhns-para}
\lim_{\alpha\rightarrow 0}\binom{n+\alpha-1}{n}\ze_n(\{1\}_k;\alpha)=\frac{\ze_{n-1}(\{1\}_{k-1})}{n}\quad \text{for }k,n\in\N.
\end{align}
Hence, taking $\alpha\rightarrow 0$ in \eqref{eq-mzvs-cb-mhs-mhss-hmzv} yields
\begin{align}\label{eq-mzvs-cb-mhs-mhss-hmzv-=0}
&\sum_{j=1}^p (-1)^{j-1} \ze(m_p,\ldots,m_{j+1})\sum_{n=1}^\infty \frac{\ze_{n-1}(\{1\}_{k-1})\ze_n^\star((m_1-1,m_2,\ldots,m_j)^\vee)}{n^2}\nonumber\\
&\quad=\sum_{i_1+\cdots+i_p=k\atop i_j\geq 0,\ \forall j}\left\{\prod\limits_{j=1}^p\binom{m_j+i_j-1}{i_j}\right\}\ze(m_p+i_p,\ldots,m_1+i_1).
\end{align}

\section{Ap\'ery-Type series with two parametric binomial coefficients}

In this section, we will use the iterated integrals to evaluate some explicit relations of Ap\'ery-type series involving two parametric binomial coefficients and Hurwitz-type multiple harmonic (star) sums in terms of Hurwitz-type MZVs. To prove the main theorems, we need the following lemma.

\begin{lem}\label{lem-dual-itearted-in} \emph{(cf. \cite[(1.6.1-2)]{KTChen1971})}
If $f_i\ (i=1,\dotsc,m)$ are integrable real functions, then the following identity holds:
\begin{equation*}
 g\left( {{f_1},\dotsc,{f_m}} \right) + {\left( {-1} \right)^m}g\left( {{f_m},\dotsc,{f_1}} \right)
 = \sum\limits_{i = 1}^{m-1} {{{\left( {-1} \right)}^{i-1}}g\left( {{f_i},{f_{i-1}}, \cdots ,{f_1}} \right)} g\left( {{f_{i + 1}},{f_{i + 2}} \cdots ,{f_m}} \right),
\end{equation*}
where $g\left( {{f_1}, \dotsc,{f_m}} \right)$ is defined by
\[g\left( {{f_1},\dotsc,{f_m}} \right): = \int_{0 < {t_m} <  \cdots < {t_1} < 1}
f_m(t_m)\cdots f_1(t_1)dt_1\cdots dt_m.\]
\end{lem}

\begin{thm}\label{thm-duality-para-apery}
For any integer $m\geq -1$ and complex numbers $\alpha,\beta\in \mathbb{C}\backslash \N_0^-$ with $\Re(\alpha),\Re(\beta)<1$, we have
\begin{align}\label{for-duality-para-apery}
&\sum\limits_{n=1}^{\infty}\frac{\binom{n+\alpha-1}{n}}{n^{m+2}\binom{n-\beta}{n}}
+(-1)^{m}\sum\limits_{n=1}^{\infty}\frac{\binom{n+\beta-1}{n}}{n^{m+2}\binom{n-\alpha}{n}}
-\sum\limits_{n=1}^{\infty}\frac{\binom{n+\alpha-1}{n}}{n^{m+2}}
-(-1)^{m}\sum\limits_{n=1}^{\infty}\frac{\binom{n+\beta-1}{n}}{n^{m+2}}\nonumber\\
&\quad=\sum\limits_{i=1}^{m+1}(-1)^{i-1}
\sum\limits_{n=1}^{\infty}\frac{\binom{n+\beta-1}{n}}{n^{i}}
\sum\limits_{n=1}^{\infty}\frac{\binom{n+\alpha-1}{n}}{n^{m+2-i}}.
\end{align}
\end{thm}
\begin{proof} For $m\geq 0$ consider the iterated integral
\begin{align}\label{para-iter-int}
I_m(\alpha,\beta):=\int_0^1{\left[\frac{1}{t}\left(\frac{1}{(1-t)^{\alpha}}-1\right)dt\right]}\cdot\underbrace{\frac{dt}{t}\cdots \frac{dt}{t}}_{m}\cdot
 \left[\frac{1}{t}\cdot\left(\frac{1}{(1-t)^{\beta}}-1\right)dt\right].
\end{align}
For $m=-1$, we set
\begin{align}\label{para-iter-int-1}
I_{-1}(\alpha,\beta):=\int_0^1 \frac{1}{t}\left(\frac{1}{(1-t)^{\alpha}}-1\right)\left(\frac{1}{(1-t)^{\beta}}-1\right)dt.
\end{align}
Note that
\begin{align}\label{beta-fun-binom}
\int_0^1 \frac{t^{n-1}}{(1-t)^\beta}dt=B(n,1-\beta)=\frac{1}{n\binom{n-\beta}{n}}\quad \text{for }\Re(\beta)<1,
\end{align}
where $B(a,b)$ is the beta function defined by \eqref{equ:betaFunction}.
Hence, applying \eqref{eq-Gen-Fun-1-x-a} and \eqref{beta-fun-binom} to \eqref{para-iter-int}, we see that for $m\geq 0$,
\begin{align}\label{exp-form-int-I}
I_m(\alpha,\beta)&=\sum\limits_{n=1}^{\infty}\binom{n+\alpha-1}{n}\int_0^1{t^{n-1}dt}\cdot\underbrace{\frac{dt}{t}\cdots \frac{dt}{t}}_{m}\cdot\left[\frac{1}{t}\cdot\left(\frac{1}{(1-t)^{\beta}}-1\right)dt\right]\nonumber\\
&=\sum\limits_{n=1}^{\infty}\frac{\binom{n+\alpha-1}{n}}{n^{m+1}}\int_0^1 {t^{n-1}\left(\frac{1}{(1-t)^{\beta}}-1\right)dt}\nonumber\\
&=\sum\limits_{n=1}^{\infty}\frac{\binom{n+\alpha-1}{n}}{n^{m+2}\binom{n-\beta}{n}}-\sum\limits_{n=1}^{\infty}\frac{\binom{n+\alpha-1}{n}}{n^{m+2}}.
\end{align}
Further,  the equation
\begin{align*}
 \int_0^1{\left[\frac{1}{t}\left(\frac{1}{(1-t)^{\alpha}}-1\right)dt\right]}\cdot\underbrace{\frac{dt}{t}\cdots \frac{dt}{t}}_{m}
=\sum\limits_{n=1}^{\infty}\frac{\binom{n+\alpha-1}{n}}{n^{m+1}}
\end{align*}
and Lemma \ref{lem-dual-itearted-in} together imply \eqref{for-duality-para-apery} for $m\geq 0$ easily.
Applying \eqref{para-iter-int-1}
we can also deduce \eqref{for-duality-para-apery} for $m=-1$  by a similar argument as in the proof above.
\end{proof}

If setting $\alpha=\beta=\frac{1}{2}$ in \eqref{for-duality-para-apery}, we get the following corollary.
\begin{cor} For any integer $m\geq -1$, we have
$$\left(1+(-1)^{m}\right)\zeta(m+2)=\sum\limits_{i=1}^{m+1}(-1)^{i-1}\sum\limits_{n=1}^{\infty}
\frac{\binom{2n}{n}}{n^{i}4^{n}}
\sum\limits_{n=1}^{\infty}\frac{\binom{2n}{n}}{n^{m+2-i}4^{n}}
+\left(1+(-1)^{m}\right)\sum\limits_{n=1}^{\infty}\frac{\binom{2n}{n}}{n^{m+2}4^{n}}.$$
\end{cor}

\begin{re} In \cite[Thm. 3.4]{WX2021}, Wang and Xu proved that for $m\in \N$,
\[\sum_{n=1}^\infty \frac{\binom{2n}{n}}{n^m 4^n}\in \Q[\log(2),\ze(2),\ze(3),\ze(4),\ldots].\]
\end{re}

\begin{thm}\label{thm-pmhs-duality-para-apery}
For any integers $m\geq -1,k,p\geq 0$ and complex numbers $\alpha,\beta\in \mathbb{C}\backslash \N_0^-$ with $\Re(\alpha),\Re(\beta)<1$, we have
\begin{align}
&\sum\limits_{n=1}^{\infty}\frac{\zeta_{n}(\{1\}_{k};\alpha)\zeta_{n}^{\star}(\{1\}_{p};1-\beta)}{n^{m+2}}
\frac{\binom{n+\alpha-1}{n}}{\binom{n-\beta}{n}}\nonumber\\
&+(-1)^{m}\sum\limits_{n=1}^{\infty}\frac{\zeta_{n}(\{1\}_{p};\beta)\zeta_{n}^{\star}(\{1\}_{k};1-\alpha)}{n^{m+2}}
\frac{\binom{n+\beta-1}{n}}{\binom{n-\alpha}{n}}\nonumber\\
&-\delta_{0,p}\sum\limits_{n=1}^{\infty}\frac{\zeta_{n}(\{1\}_{k};\alpha)}{n^{m+2}}
\binom{n+\alpha-1}{n}-(-1)^{m}\delta_{0,k}\sum\limits_{n=1}^{\infty}\frac{\zeta_{n}(\{1\}_{p};\beta)}{n^{m+2}}
\binom{n+\beta-1}{n}\nonumber\\
&=\sum\limits_{i=1}^{m+1}(-1)^{i-1}\left\{\sum\limits_{n=1}^{\infty}
\frac{\zeta_{n}(\{1\}_{p};\beta)}{n^{i}}
\binom{n+\beta-1}{n}\right\}
\left\{\sum\limits_{n=1}^{\infty}
\frac{\zeta_{n}(\{1\}_{k};\alpha)}{n^{m+2-i}}
\binom{n+\alpha-1}{n}\right\}.
\end{align}
\end{thm}
\begin{proof}
Differentiating \eqref{for-duality-para-apery} $k$ times with respect to $\alpha$ and $p$ times with respect to $\beta$ yields the desired result with the help of \eqref{diff-para-binom-mhss} and \eqref{eq-diff-pcb}.
\end{proof}

Setting $(\alpha,\beta)\rightarrow (0,0)$ and $(0,1/2)$, respectively, yields the following two corollaries.
\begin{cor} For any integers $m\geq -1$ and $k,p\geq 1$, we have
\begin{align}
&\sum\limits_{n=1}^{\infty}\frac{\zeta_{n-1}(\{1\}_{k-1})\zeta_{n}^{\star}(\{1\}_{p})}{n^{m+3}}
+(-1)^{m}\sum\limits_{n=1}^{\infty}\frac{\zeta_{n-1}(\{1\}_{p-1})\zeta_{n}^{\star}(\{1\}_{k})}{n^{m+3}}\nonumber\\
&\quad=\sum\limits_{i=1}^{m+1}(-1)^{i-1}\zeta(i+1;\{1\}_{p-1})\zeta(m+3-i;\{1\}_{k-1}).
\end{align}
\end{cor}

\begin{cor} For any integers $m\geq -1,k\geq 1$ and $p\geq 0$, we have
\begin{align}
&\sum\limits_{n=1}^{\infty}\frac{\zeta_{n-1}(\{1\}_{k-1})t_{n}^{\star}(\{1\}_{p})}{n^{m+3}}
\frac{4^{n}}{\binom{2n}{n}}
+(-1)^{m}\sum\limits_{n=1}^{\infty}\frac{t_{n}(\{1\}_{p})\zeta_{n}^{\star}(\{1\}_{k})}{n^{m+2}}
\frac{\binom{2n}{n}}{4^{n}}
-\delta_{0,p}\zeta(m+3,\{1\}_{k-1})\nonumber\\
&\quad=\sum\limits_{i=1}^{m+1}(-1)^{i-1}
\left\{\sum\limits_{n=1}^{\infty}
\frac{t_{n}(\{1\}_{p})}{n^{i}}\frac{\binom{2n}{n}}{4^{n}}\right\}
\zeta(m+3-i,\{1\}_{k-1}).
\end{align}
\end{cor}

\begin{thm} For any integer $m\geq -1$ and complex numbers $\alpha\in \mathbb{C}\backslash \N_0^-,\beta\in \mathbb{C}$ with $\Re(\beta)<1$, we have
\begin{align}\label{exp-formul-t-para-binom}
\sum\limits_{n=1}^{\infty}\frac{\binom{n+\alpha-1}{n}}{n^{m+2}\binom{n-\beta}{n}}
=\alpha\sum\limits_{n=1}^{\infty}\frac{\zeta_{n-1}(\{1\}_{m+1};1-\beta)}{(n-\beta-\alpha)(n-\beta)}.
\end{align}
\end{thm}
\begin{proof} We only prove that \eqref{exp-formul-t-para-binom} holds for $m\geq 0$. Applying $t\rightarrow 1-t$ in \eqref{para-iter-int},
\begin{align*}
&I_m(\alpha,\beta)=\int_0^1{\left[\frac{1}{1-t}\left(\frac{1}{t^{\beta}}-1\right)dt\right]}
\cdot\underbrace{\frac{dt}{1-t}\cdots \frac{dt}{1-t}}_{m}\cdot
 \left[\frac{1}{1-t}\cdot\left(\frac{1}{t^{\alpha}}-1\right)dt\right]\\
&=\lim_{x \to 1}\int_0^x{\left[\frac{1}{1-t}\left(\frac{1}{t^{\beta}}-1\right)dt\right]}
\cdot\underbrace{\frac{dt}{1-t}\cdots \frac{dt}{1-t}}_{m}\cdot
 \left[\frac{1}{1-t}\cdot\left(\frac{1}{t^{\alpha}}-1\right)dt\right]\\
&=\lim_{x \to 1}\left\{ \begin{array}{l}
\int_0^x{\frac{dt}{(1-t)t^{\beta}}}\cdot\underbrace{\frac{dt}{1-t}\cdots \frac{dt}{1-t}}_{m}\cdot
\frac{dt}{(1-t)t^{\alpha}}-\int_0^x{\frac{dt}{(1-t)t^{\beta}}}\cdot\underbrace{\frac{dt}{1-t}\cdots\frac{dt}{1-t}}_{m}\cdot\frac{dt}{1-t}\\
-\int_0^x{\frac{dt}{1-t}}\underbrace{\frac{dt}{1-t}\cdots\frac{dt}{1-t}}_{m}\cdot\frac{dt}{(1-t)t^{\alpha}}
+\int_0^x{\frac{dt}{1-t}}\underbrace{\frac{dt}{1-t}\cdots\frac{dt}{1-t}}_{m}\cdot\frac{dt}{1-t}
 \end{array} \right\}\\
&=\lim_{x \to 1}\left\{ \begin{array}{l}
\sum\limits_{n_{1}>\cdots>n_{m+2}>0}\frac{x^{n_{1}-\beta-\alpha}}{(n_{1}-\beta-\alpha)\prod_{j=2}^{m+2}(n_{j}-\beta)}-\sum\limits_{n_{1}>\cdots>n_{m+2}>0}\frac{x^{n_{1}-\beta}}{\prod_{j=1}^{m+2}
(n_{j}-\beta)}\\
-\sum\limits_{n_{1}>\cdots>n_{m+2}>0}\frac{x^{n_{1}-\alpha}}{(n_{1}-\alpha)n_2\cdots n_{m+2}}
+\sum\limits_{n_{1}>\cdots>n_{m+2}>0}\frac{x^{n_{1}}}{n_{1}\cdots n_{m+2}}
\end{array} \right\}\\
&=\sum\limits_{n=1}^{\infty}\zeta_{n-1}(\{1\}_{m+1};1-\beta)\left(\frac{1}{n-\beta-\alpha}-\frac{1}{n-\beta}\right)
-\sum\limits_{n=1}^{\infty}\zeta_{n-1}(\{1\}_{m+1})\left(\frac{1}{n-\alpha}-\frac{1}{n}\right)\\
&=\alpha\sum\limits_{n=1}^{\infty}\frac{\zeta_{n-1}(\{1\}_{m+1};1-\beta)}{(n-\beta-\alpha)(n-\beta)}
-\alpha\sum\limits_{n=1}^{\infty}\frac{\zeta_{n-1}(\{1\}_{m+1})}{n(n-\alpha)}.
\end{align*}
Hence, applying \eqref{exp-form-int-I} yields
\begin{align*}
\sum\limits_{n=1}^{\infty}\frac{\binom{n+\alpha-1}{n}}{n^{m+2}\binom{n-\beta}{n}}
-\sum\limits_{n=1}^{\infty}\frac{\binom{n+\alpha-1}{n}}{n^{m+2}}
=\alpha\sum\limits_{n=1}^{\infty}\frac{\zeta_{n-1}(\{1\}_{m+1};1-\beta)}{(n-\beta-\alpha)(n-\beta)}
-\alpha\sum\limits_{n=1}^{\infty}\frac{\zeta_{n-1}(\{1\}_{m+1})}{n(n-\alpha)}.
\end{align*}
Further, \eqref{eq-not-mhs-1} gives rise to the formula
\[\sum\limits_{n=1}^{\infty}\frac{\binom{n+\alpha-1}{n}}{n^{m+1}}
=\alpha\sum\limits_{n=1}^{\infty}\frac{\zeta_{n-1}(\{1\}_{m})}{n(n-\alpha)}\,,
\]
which leads to \eqref{exp-formul-t-para-binom} immediately.
\end{proof}

\begin{thm}\label{thm-t-para-binom-t-mhss}
For any integers $m\geq -1,k,p\geq 0$ and complex numbers $\alpha\in \mathbb{C}\backslash \N_0^-,\beta\in \mathbb{C}$ with $\Re(\beta)<1$, we have
\begin{align}\label{for-t-para-binom-t-mhss}
&\sum\limits_{n=1}^{\infty}\frac{\zeta_{n}(\{1\}_{k};\alpha)\zeta_{n}^{\star}(\{1\}_{p};1-\beta)}{n^{m+2}}
\frac{\binom{n+\alpha-1}{n}}{\binom{n-\beta}{n}}\nonumber\\
&=\sum\limits_{i_{1}+\cdots+i_{m+2}=p\atop i_{1}, \ldots ,i_{m+2}\geq0}
\binom{i_{1}+k}{k}\left\{\zeta(i_1+k+1,i_2+1,\ldots,i_{m+2}+1;1-\alpha-\beta,\{1-\beta\}_{m+1})\atop -\delta_{0,k}\zeta(i_1+k+1,i_2+1,\ldots,i_{m+2}+1;\{1-\beta\}_{m+2}) \right\},
\end{align}
where if $i_1=k=0$ then
\begin{align}
&\zeta(1,i_2+1,\ldots,i_{m+2}+1;1-\alpha-\beta,\{1-\beta\}_{m+1})-\zeta(1,i_2+1,\ldots,i_{m+2}+1;\{1-\beta\}_{m+2})\nonumber\\
&=\alpha\sum_{n=1}^\infty \frac{\zeta_{n-1}(i_{2}+1,\ldots,i_{m+2}+1;1-\beta)}{(n-\beta)(n-\alpha-\beta)}.
\end{align}
\end{thm}
\begin{proof}
Differentiating \eqref{exp-formul-t-para-binom} $k$ times with respect to $\alpha$ and $p$ times with respect to $\beta$ yields the desired result with the help of \eqref{diff-para-binom-mhss} and \eqref{eq-diff-pcb}.
\end{proof}

If $(\alpha,\beta)\rightarrow (0,0)$ in \eqref{for-t-para-binom-t-mhss}, we obtain the well-known result of Arakawa--Kaneko zeta values \cite[Thm. 9(i)]{AM1999}.
Letting $(\alpha,\beta)\rightarrow (0,1/2), (1/2,1/2)$ and $(1/2,0)$, we get the following corollaries.
\begin{cor} For any integers $m\geq -1, k\geq 1$ and $p\geq 0$, we have
\begin{align}
&\sum\limits_{n=1}^{\infty}\frac{\zeta_{n-1}(\{1\}_{k-1})t_{n}^{\star}(\{1\}_{p})}{n^{m+3}}
\frac{4^{n}}{\binom{2n}{n}}\nonumber\\
&=2^{k+m+2}\sum\limits_{i_{1}+\cdots+i_{m+2}=p\atop i_{1}, \ldots, i_{m+2}\geq0}
\binom{i_{1}+k}{k} t(i_{1}+k+1,i_{2}+1,\ldots,i_{m+2}+1).
\end{align}
\end{cor}

\begin{cor} For any integers $m, k,p\geq 0$, we have
\begin{align}
&2^{k+p}\sum\limits_{n=1}^{\infty}\frac{t_{n}(\{1\}_{k})t_{n}^{\star}(\{1\}_{p})}{n^{m+2}}\nonumber\\
&=\sum\limits_{i_{1}+\cdots+i_{m+2}=p\atop i_{j}\geq0,\forall j} 2^{p-i_{1}}\binom{i_{1}+k}{k}
\left\{ {\begin{array}{*{20}{c}}
\sum\limits_{n=1}^{\infty}\frac{t_{n}(i_{2}+1,\ldots,i_{m+2}+1)}{n^{i_{1}+k+1}}
   \ \ \ \ \ \ \ \ \ \ \ \ \ \ \ \ \ \ \ \ \ \ \ \ \ \;\;\; \;\;\;\;\ \ \ \ {k\geq 1,} \\
\sum\limits_{n=1}^{\infty}t_{n}(i_{2}+1,\ldots,i_{m+2}+1)\left(\frac{1}{n^{i_{1}+1}}
  -\frac{1}{(n-1/2)^{i_{1}+1}}\right) \;\;\;\;{k=0.}  \\
\end{array} } \right.
\end{align}
\end{cor}

\begin{cor} For any integers $m\geq -1, k\geq 1$ and $p\geq 0$, we have
\begin{align}
 &2^{k}\sum\limits_{n=1}^{\infty}\frac{t_{n}(\{1\}_{k})\zeta_{n}^{\star}(\{1\}_{p})}{n^{m+2}}
\frac{\binom{2n}{n}}{4^{n}} \nonumber\\
&=\sum\limits_{i_{1}+\cdots+i_{m+2}=p\atop i_{j}\geq0,\forall j}\binom{i_{1}+k}{k}
\left\{ {\begin{array}{*{20}{c}}
\sum\limits_{n=1}^{\infty}\frac{\zeta_{n-1}(i_{2}+1,\ldots,i_{m+2}+1)}{(n-1/2)^{i_{1}+k+1}}
   \ \ \ \ \ \ \ \ \ \ \ \ \ \ \ \ \ \ \ \ \ \ \ \ \ \;\;\; \;\;\;\;\ \ \ \ {k\geq 1,} \\
\sum\limits_{n=1}^{\infty}\zeta_{n-1}(i_{2}+1,\ldots,i_{m+2}+1)\left(\frac{1}{(n-1/2)^{i_{1}+1}}
-\frac{1}{n^{i_{1}+1}}\right) \;\;\;\;{k=0.}
\end{array} } \right.
\end{align}
\end{cor}

\medskip\noindent
{\bf Question.} Let $r,p$ and $m$ be nonnegative integers. Is it true that for any two compositions of positive integers $\bfk=(k_1,\ldots,k_r)\in \N^r$ and $\bfl=(l_1,\ldots,l_p)\in\N^p$, the following parametric Ap\'{e}ry-type series
\begin{align}
\sum\limits_{n=1}^{\infty}\frac{\zeta_{n}(k_1,\ldots,k_r;\alpha)\zeta_{n}^{\star}(l_1,\ldots,l_p;1-\beta)}{n^{m+2}}
\frac{\binom{n+\alpha-1}{n}}{\binom{n-\beta}{n}}
\end{align}
can be expressed by a linear combination of Hurwitz-type multiple zeta values?

\medskip
{\bf Acknowledgments.}  Masanobu Kaneko is supported by the JSPS KAKENHI Grant Numbers JP16H06336, JP21H04430. Weiping Wang is supported by the Zhejiang Provincial Natural Science Foundation of China (Grant No. LY22A010018). The corresponding author Ce Xu is supported by the National Natural Science Foundation of China (Grant No. 12101008), the Natural Science Foundation of Anhui Province (Grant No. 2108085QA01) and the University Natural Science Research Project of Anhui Province (Grant No. KJ2020A0057). Jianqiang Zhao is supported by the Jacobs Prize from The Bishop's School.

\medskip

{\bf Data Availability Statement.} There is no data associated with the paper.

\medskip
{\bf Declarations of competing interest}: none.

\end{document}